\documentclass[11pt]{amsart}

\usepackage{amsmath}
\usepackage{amsthm}
\usepackage{hyperref}
\usepackage[margin=1.0in]{geometry}
\usepackage[T1]{fontenc}
\usepackage[applemac]{inputenc}
\usepackage{MnSymbol}
\usepackage{xcolor}

\numberwithin{equation}{section}

\newtheorem{prop}{Proposition}
\newtheorem{lemma}[prop]{Lemma}

\newtheorem{thm}[prop]{Theorem}
\newtheorem{cor}[prop]{Corollary}
\newtheorem{conj}[prop]{Conjecture}
\numberwithin{prop}{section}

\theoremstyle{definition}
\newtheorem{defn}[prop]{Definition}

\newtheorem{rmk}[prop]{Remark}

\DeclareSymbolFont{script}{U}{eus}{m}{n}
\DeclareSymbolFontAlphabet{\mathscr}{script}
\DeclareMathSymbol{\Wedge}{0}{script}{"5E}
\DeclareMathAlphabet{\mathrmsl}{OT1}{cmr}{m}{sl}

\newcommand{\ang}{\theta}
\newcommand{\del}{\partial}
\newcommand{\delb}{\bar{\partial}}

\newcommand{\brs}[1]{\left| #1 \right|}
\newcommand{\G}{{\mathbb G}}

\newcommand{\gD}{\Delta}
\newcommand{\gs}{\sigma}

\newcommand{\gk}{\kappa}
\newcommand{\gl}{\lambda}

\newcommand{\tI}{\tilde{I}}

\newcommand{\ga}{\alpha}
\newcommand{\gb}{\beta}
\newcommand{\gL}{\Wedge}
\newcommand{\N}{\nabla}

\newcommand{\WW}{\mathcal W}

\renewcommand{\bar}[1]{\overline{#1}}

\renewcommand{\i}{\sqrt{-1}}
\newcommand{\Pol}{{\mathrm P}}
\newcommand{\X}{{\mathbf X}} 
\newcommand{\Y}{{\mathbf Y}} 

\newcommand{\cH}{{\mathcal H}}
\newcommand{\I}{{\bf I}} 
\newcommand{\M}{{\bf M}} 
\newcommand{\E}{{\mathcal E}} 
\newcommand{\PSH}{\mathrm{PSH}} 

\renewcommand{\part}{\del}

\newcommand{\hook}{\righthalfcup}

\newcommand{\be}{\bar{e}}

\newcommand{\IP}[1]{\left<#1\right>}

\newcommand{\som}{\omega_0}
\newcommand{\pg}{g}
\newcommand{\kom}{\bar{\omega}}
\newcommand{\gw}{\omega}%
\newcommand{\pom}{\omega}
\newcommand{\kg}{{\bar g}} 
\newcommand{\Fo}{\omega_0} 
\newcommand{\kJ}{{\bar J}} 
\newcommand{\kI}{{\bar I}} 

\newcommand{\tih}{{h}}

\DeclareMathOperator{\Rc}{Rc}

\DeclareMathOperator{\tr}{tr}

\DeclareMathOperator{\Vol}{Vol}

\renewcommand{\geq}{\geqslant}

\newcommand{\R}{{\mathbb R}}
\newcommand{\C}{{\mathbb C}}
\newcommand{\Z}{{\mathbb Z}}

\newcommand{\T}{{\mathbb T}}

\newcommand{\tor}{{\mathfrak t}}

\newcommand{\kal}{\bar{\alpha}}
\newcommand{\cL}{{\mathcal L}}
\newcommand{\cS}{{\mathcal S}}

\newcommand{\Aut}{\mathrmsl{Aut}}

\newcommand{\trace}{\mathop{\mathrm{tr}}\nolimits}
\newcommand{\ip}[1]{\langle #1 \rangle}
\renewcommand{\d}{{\mathrmsl d}}
\newcommand{\Hess}{\mathrm{Hess}}

\newcommand{\Fi}{F}
\newcommand{\La}{{\mathrm L}}
\def\th/#1#2{{#1}^{#2}}

\numberwithin{prop}{section}

\theoremstyle{definition}

\newtheorem{rem}[prop]{Remark}

\newtheorem{convention}{Convention}

\def\fnote#1{}

\def\be{\begin{equation}}
\def\ee{\end{equation}}
\def\bea{\begin{eqnarray*}}
\def\eea{\end{eqnarray*}}

\newenvironment{bulletlist}{\begin{list}{\labelitemi}%
{\setlength{\leftmargin}{\parindent}\def
\makelabel ##1{\hss \llap {\upshape ##1}}}}{\end{list}}
\date{\today}

\begin{document}

\title[Generalized K\"ahler-Ricci flow on toric Fano varieties]{Generalized K\"ahler-Ricci flow on toric Fano varieties}

\begin{abstract} We study the generalized K\"ahler-Ricci flow with initial data of symplectic type, and show that this condition is preserved.  In the case of a Fano background with toric symmetry, we establish global existence of the normalized flow.  We derive an extension of Perelman's entropy functional to this setting, which yields convergence of nonsingular solutions at infinity.  Furthermore, we derive an extension of Mabuchi's $K$-energy to this setting, which yields weak convergence of the flow.
\end{abstract}

\author{Vestislav Apostolov}
\address{Vestislav Apostolov\\ D\'epartment de mathe\'ematiques\\
         Universit\'e du Qu\'ebec \`a Montr\'eal\\
         Case postale 8888, succursale centre-ville
         Mongtr\'eal (Qu\'ebec) H3C 3P8}
\email{\href{mailto:apostolov.vestislav@uqam.ca}{apostolov.vestislav@uqam.ca}}

\author{Jeffrey Streets}
\address{Jeffrey Streets\\Rowland Hall\\
         University of California\\
         Irvine, CA 92617}
\email{\href{mailto:jstreets@uci.edu}{jstreets@uci.edu}}

\author{Yury Ustinovskiy}
\address{Yury Ustinovskiy\\Courant Institute of Mathematical Sciences\\
New York University\\
251 Mercer Street\\
New York, NY, 10012-1185} 
\email{\href{mailto:yura.ust@nyu.edu}{yura.ust@nyu.edu}}

\thanks{JS was supported by the NSF via DMS-1454854.  VA was supported in part by an NSERC Discovery Grant. The authors are grateful to Tamas Darvas, Xiaohua Zhu, and an anonymous referee for their comments on the manuscript.}
\maketitle

\section{Introduction}

Generalized K\"ahler (GK) structures first appeared through investigations into supersymmetric sigma models \cite{Gates}, and were rediscovered in
a purely mathematical context in the work of Gualtieri~\cite{Gualtieri-PhD} and  Hitchin~\cite{HitchinGCY}. They have recently
attracted interest in both the physics and mathematical communities as natural generalizations of K\"ahler structures.  We will focus here entirely on the so called {\it biHermitian} description of generalized K\"ahler geometry (cf. \cite{AGG,Gates}). Thus,  a
generalized K\"ahler manifold is a smooth manifold $M$ with a triple $(g, I, J)$ consisting
of two integrable almost-complex structures,  $I$ and $J$,  together with a
Riemannian metric $g$ which is Hermitian with respect to both, such that the K\"ahler forms $\gw_I$ and
$\gw_J$ satisfy
\begin{align}\label{GK}
 \d^c_I \gw_I = H = - \d^c_J \gw_J, \qquad \d H = 0,
\end{align}
where the first equation defines $H$, and $\d^c_I = \i (\delb_I - \del_I), \, \d^c_J=\i(\delb_J-\del_J)$.

The generalized K\"ahler-Ricci flow (GKRF) was introduced by the second named author and Tian in \cite{STGK} as a tool for constructing canonical metrics and understanding existence and moduli problems in generalized K\"ahler geometry, extending the deeply developed theory in K\"ahler geometry.  Prior global existence and convergence results for GKRF have appeared in for instance \cite{AS,SNDG, SPCFSTB}.  In this paper we consider {\it normalized generalized K\"ahler Ricci flow} (NGKRF), expressed as a flow of $\d\d_J^c$-closed positive definite $(1,1)$-forms $\omega_{J}$ on the complex manifold  $(M, J)$, together with a flow of $I$, as (cf. \cite{STGK})
\begin{gather}
\begin{split}\label{e:NGKRF}
\frac{\partial}{\partial t} \omega_J =&\ - 2\Big(\rho^{B}_J\Big)^{1,1}  + 2\lambda \omega_J, \qquad \frac{\del}{\del t} I = \cL_{\theta_I^{\sharp} - \theta_J^{\sharp}} I,
\end{split}
\end{gather}
where $\lambda$ is a fixed real constant  and $\Big(\rho_J^{B}\Big)^{1,1}$ denotes the $(1,1)$-part with respect to $J$  of the Bismut-Ricci form $\rho_J^B \in 2\pi c_1(M, J)$ of $(\gw_J, J)$.  By~\cite{STGK}, as far as a solution to \eqref{e:NGKRF} exists,
the triple $(g_t, I_t, J)$ defines a GK structure  in the sense of \eqref{GK}.  Notice that if the initial GK structure  $(g, I, J)$ was K\"ahler  (i.e.  $\d^c_J\omega_J=-H=0$ and $I=J$ in \eqref{GK}),  the Bismut-Ricci form $\rho^B$ is just the usual K\"ahler-Ricci form of the K\"ahler manifold $(g, J)$, and \eqref{e:NGKRF} reduces to the normalized K\"ahler Ricci flow.   In this setting Cao showed \cite{Cao} smooth global existence.  The convergence at infinity is a very subtle issue, and many results have appeared in recent years \cite{CSW, DS, Phongetal, TZZZ, TZ3f, TZ-07, TZ-13}.

In this paper, we will study NGKRF in the case when $(M, J)$ is a smooth toric Fano variety of complex dimension $m$, assuming that the initial GK structure $(g, I, J)$ is compatible with a symplectic form $F \in 2 \pi c_1(M, J)$,  and  is invariant under the action of a maximal compact torus $\T$ in the automorphism group $\Aut(M, J)$ of $(M,J)$.  More precisely, suppose that  
$F$ is a symplectic $2$-form on $M$ which tames $J$, i.e. whose $(1,1)$-part  $\omega_J := \big(F \big)^{1,1}_J$  is positive definite. We then define a Riemannian metric $g$, a $2$-form $b$ and  an almost complex structure $I$ by
\begin{align}\label{GK-symplectic}
g:= -(FJ)^{\rm sym},\qquad b:=-(FJ)^{\rm skew}, \qquad I := - F^{-1}J^*F.
\end{align}
One can show (cf. \S \ref{s:symplectic-type}) that when $I$ is also integrable,  $(g, I, J)$ is a GK structure with $H= d b$.  Furthermore, in this case we say that $(g, I, J)$ is an $F$-{\it compatible} GK structure, or a GK structure of {\it symplectic type},  referring to the fact~\cite{Gualtieri-CMP} that one of the corresponding generalized complex structures on $TM \oplus T^*M$ is determined completely by the symplectic $2$-form $F$.

To better understand these structures, first recall that a key feature of GK geometry, observed by Hitchin~\cite{HitchinPoisson}  (cf. \cite{AGG, Pontecorvo} for the $4$-dimensional
case) is that there are naturally associated holomorphic Poisson structures.  More precisely,
the tensors
\begin{align}\label{GK-Poisson}
\gs = [I,J] g^{-1}, \qquad \sigma_I = \sigma - \sqrt{-1} I \sigma, \qquad \sigma_J = \sigma - \sqrt{-1} J \sigma
\end{align}
define respectively a real Poisson structure and holomorphic Poisson structures with respect to $I$ and $J$.  Conversely, given a compact K\"ahler manifold $(M, J, \kg, \kom)$ and a non-trivial holomorphic Poisson structure $\mu \in H^0(M, \gL^2(T^{1,0}M))$, the are deformation results~\cite{HitchinJSG,Goto-AM,Gualtieri-Poisson, Gualtieri-Hamiltonian} producing non-K\"ahler GK structures  $(g_t, I_t, J)$ with holomorphic Poisson structure $\sigma_{J_t}=t\mu$ for $|t|< \varepsilon$. This takes a particularly nice and explicit form on a K\"ahler toric variety $(M, J, \T)$.  Recent work of Boulanger~\cite{boulanger} constructs invariant toric GK structures by deforming the complex structure $J$ using $A\in \gL^2\tor$, referred to here as {\it deformations of type $A$}.  More recently, Y. Wang~\cite{W1,W2} constructed toric GK structures by deforming $\bar{\gw}$ to a symplectic form $F$ using $B \in \gL^2\tor$, referred to here as {\it deformations of type $B$}.  It follows from the results in \cite{W2} (cf. Proposition \ref{p:toric-GK} below) that, up to pull-back by $\T$-equivariant diffeomorphism, any $\T$-invariant GK structure $(g, I, J)$ of symplectic type on $(M, J)$ with corresponding symplectic $2$-form $F$, is obtained  from a $\T$-invariant K\"ahler structure $(\kg, \kJ)$ by deformations of type $A$ and $B$.  We furthermore observe in Proposition \ref{p:toric-Poisson} a geometric interpretation of $A$ and $B$-type deformations in terms of the holomorphic Poisson structures.

Building upon this, we show in \S \ref{s:symplectic-type} that NGKRF interacts with the associated Poisson structures in a predictable way, and thus is a natural tool for investigating questions on the global moduli of generalized K\"ahler structures on, as well as the underlying Poisson geometry of, Fano varieties.  We formulate certain precise conjectures coming from the formal picture of NGKRF in \S \ref{s:symplectic-type}.  As a fundamental first step in this direction, we establish the sharp global existence of the flow in the case of toric symmetry.

\begin{thm} \label{t:mainthm} Let $(M, J, \T)$ be a smooth toric Fano  variety and $F\in 2\pi c_1(M, J)$ a $\T$-invariant symplectic $2$-form which defines a $\T$-invariant generalized K\"ahler structure $(g, I, J)$ of symplectic type on $M$. Then, the solution of normalized generalized K\"ahler Ricci flow \eqref{e:NGKRF} with this initial data exists on $[0,\infty)$.
\end{thm}

The key geometric observation behind Theorem \ref{t:mainthm} is that in the case of a Fano background with toric symmetry, the normalized generalized K\"ahler Ricci flow \eqref{e:NGKRF} reduces to a flow of  $\T$-invariant \emph{K\"ahler} metrics $\kom_t \in 2\pi c_1(M, J)$.  Using the Abreu--Guillemin reduction~\cite{Abreu2,guillemin}, these are then determined by a family of strictly convex smooth functions $\phi_t$ on $\R^m$,  satisfying
\begin{align}\label{toric-GKRF-Kahler}
\frac{\partial}{\partial t} \phi = \log  {\det}  \Big( \big({\rm Hess}\ \phi \big)^{-1} + \sqrt{-1}e^{-2t}B\Big)^{-1} + 2\phi,
\end{align}
where $B$ is a skew-symmetric matrix corresponding to the $B$-type deformation described above.  Interestingly, this reduced scalar equation is independent of $A$-type deformations.  Also, when $B=0$, it coincides with the PDE describing the reduction of the K\"ahler Ricci flow on a smooth toric Fano variety, studied in~\cite{Zhu}.  We emphasize here that the scalar reduction in the K\"ahler-Ricci flow case is an adaptation to the toric setting of the general scalar reduction, itself a consequence of the transgression formula for the Ricci curvature.  In our setting the derivation of this scalar reduction is more subtle.  In GK geometry, the local Ricci potential relies on delicate constructions in generalized geometry (cf. \cite{GRFBook} Definition 8.15), and is not always easy to make explicit.  In \S \ref{s:toric-GK}, after a careful buildup of the differential geometry of toric generalized K\"ahler structure, we derive the Bismut-Ricci potential.  However, the relevant formula involves the evolving complex structure $I$, and therefore a scalar reduction of the flow is not immediate.  Again subtle structure of toric GK geometry enters to give the scalar reduction to (\ref{toric-GKRF-Kahler}) in Proposition \ref{p:phireduction}.  Given this scalar reduction, through a series of estimates based on the maximum principle we derive a priori $L^{\infty}$ estimates for the associated metric in terms of a $C^0$ bound for $\phi$.  Using the generalization of Evans-Krylov $C^{2,\alpha}$/Calabi-Yau $C^3$ estimates for pluriclosed flow (\cite{JS,SBIPCF}), we obtain $C^{\infty}$ estimates for $\phi$ in terms of a $C^0$ estimate for $\phi$.  Since $\phi$ grows at worst exponentially by the maximum principle, we thus conclude the long time existence.

One expects the convergence at infinity to be a delicate issue, and we prove two partial results in this direction.  In the toric KRF setting the convergence was established by Zhu \cite{Zhu} using delicate estimates derived using Perelman's entropy, Mabuchi's $K$-energy and estimates from convex geometry.  Outside of the toric setting, convergence assuming the existence of a K\"ahler-Ricci soliton was shown by Tian-Zhu \cite{TZ-07, TZ-13}.  First we prove convergence of nonsingular solutions without the toric symmetry hypothesis, where so far there is not a scalar reduction of the flow.  The key point is a Perelman-type entropy monotonicity, which yields that the only smooth self-similar limits of NGKRF are in fact K\"ahler-Ricci solitons, and also shows a uniform $\gk$-noncollapsing result for NGKRF on Fano manifolds.

\begin{thm} \label{t:genconv} Let $(M, g, I, J)$ be a generalized K\"ahler structure of symplectic type with respect to $F$.  Let $(g_t, I_t, J)$ denote the solution to NGKRF with this initial data.  Then
\begin{enumerate}
\item The metric $g_t$ is uniformly $\gk$-noncollapsed for all times the flow exists.
\item Suppose $(M, J)$ is Fano, $F \in 2 \pi c_1(M, J)$.  Suppose the solution exists on $[0,\infty)$ with uniformly bounded curvature.  Then any sequence of times $\{t_j\} \to \infty$ admits a subsequence such that $(g_{t_j}, I_{t_j}, J)$ converges in the $C^{\infty}$-Cheeger-Gromov topology to $(g_{KRS}, J_{\infty}, J_{\infty})$, where $g_{KRS}$ is a K\"ahler-Ricci soliton.
\end{enumerate}
\end{thm}
\noindent In the convergence statement above, the complex structure $J_{\infty}$ may not be in general biholomorphic to $J$, and as is the case in K\"ahler-Ricci flow one expects this subtle issue to be related to $K$-stability of $(M, J)$. In the case when the initial structure $(g, I, J)$  has a \emph{toric symmetry}, the scalar reduction \eqref{toric-GKRF-Kahler} of the GKRF and the results in~\cite{Zhu} motivate us to conjecture that $J_{\infty} \cong J$ (see Conjecture~\ref{main-conjecture} below).

Returning to the setting of toric symmetry, we are able to prove a certain kind of weak, but unconditional, convergence.  First, adapting ideas from general GIT theory, we are able to define an extension of the Mabuchi $K$-energy to toric generalized K\"ahler structures, which we show is monotone along NGKRF.  Using this monotonicity and exploiting the variational approach of  \cite{BBEGZ}, we obtain a weak convergence.

\begin{thm} \label{t:convthm} Let $(M, J, \T)$ be a smooth toric Fano  variety which admits a K\"ahler-Einstein metric, and $F\in 2\pi c_1(M, J)$ a $\T$-invariant symplectic $2$-form which defines a $\T$-invariant generalized K\"ahler structure $(g, I, J)$ of symplectic type on $M$. Let $(g_t, I_t,J), t \in [0, \infty)$ be  the global solution of the normalized generalized K\"ahler Ricci flow \eqref{e:NGKRF} with this initial data guaranteed by Theorem~\ref{t:mainthm}, and $\kom_t$ the corresponding family of toric K\"ahler metrics defined by \eqref{toric-GKRF-Kahler}. Then,  there exists a sequence of times $j\to \infty$  and automorphisms $\tau_j \in \T_{\C}$,  such that $\tau_j^*(\kom_j)$ converges, with respect to the distance $d_1$,  to a positive $(1,1)$-current $\kom_{\infty}$ of maximal Monge-Amp\`ere mass and finite energy on $(M, J)$. \end{thm}

\noindent The precise definitions for the convergence are given in Section~\ref{s:weakconvergence} below, but in the toric case we consider it can be equivalently characterized~\cite{Darvas-et-al, guedj-toric}  in terms of the $L^1$ convergence over the Delzant polytope of $(M, F, \T)$ of the Legendre transforms $u_j$ of the convex functions $\phi_j$ satisfying \eqref{toric-GKRF-Kahler}, up to the addition to $u_j$ of affine-linear functions. 

Here is an outline of the rest of the paper.  In \S \ref{s:symplectic-type} we recall fundamental properties of GK structures and GKRF, focusing on the symplectic-type case.  In \S \ref{s:toric-GK} we discuss toric GK structures, unify the discussions of \cite{boulanger} and \cite{W1}, \cite{W2}, and explicitly describe the relevant associated Poisson tensors.  This leads to a derivation of the underlying K\"ahler metric and the associated metric potential and Bismut-Ricci potential.  With these geometric results in place, we give the proof of Theorem \ref{t:mainthm} in \S \ref{s:LTE}, relying principally on the scalar reduction of the flow described above.  In the final \S \ref{s:MF}, we obtain the monotonicity of the Perelman-type and Mabuchi-type functionals, and give the proofs of Theorems~\ref{t:genconv}  and \ref{t:convthm}.

\section{The generalized K\"ahler-Ricci flow in the symplectic-type case}\label{s:symplectic-type}
\subsection{Curvature identities for Hermitian metrics}

In this subsection we recall the definitions of the Bismut-Ricci and the Chern-Ricci forms on a Hermitian manifold $(M, g, J)$, extending the notion of the K\"ahler-Ricci form of a K\"ahler manifold. 

\begin{defn} \label{Bismutdef} Let $(M, g, J)$ be a  Hermitian  manifold with  fundamental $2$-form $\omega_J = gJ$.  The \emph{Bismut connection} is the Hermitian connection defined
by
\begin{align} \label{Bismutformula}
g(\N^{B}_X Y, Z) =&\ g(\N_X Y, Z) - \tfrac{1}{2} \d^c \gw_J(X,Y,Z),
\end{align}
where $\nabla$ stands for the Riemannian connection of $g$.  The \emph{Chern connection} is the Hermitian connection defined
by
\begin{align} \label{Chernformula}
g(\N^{C}_X Y, Z) =&\  g(\N_X Y, Z) + \tfrac{1}{2} \d^c \gw_J (X,JY,JZ).
\end{align}
\end{defn}

\begin{defn} \label{Riccidefs} Let $(M, g, J)$ be a complex
manifold with a Hermitian  metric $g$.  The \emph{Bismut-Ricci curvature} and
\emph{Chern-Ricci curvature} are defined by
\begin{align*}
 \rho^B(X,Y) =&\ \tfrac{1}{2} \sum_{i=1}^{2m}R^B(X,Y,e_i, J e_i),
\qquad  \rho^C(X,Y) =\ \tfrac{1}{2} \sum_{i=1}^{2m}R^C(X,Y,e_i, J e_i),
\end{align*}
where $R^B$ and $R^C$ are the curvature tensors associated to the
Bismut \cite{Bismut} and Chern connections respectively, and $\{e_i\}$ is an orthonormal 
basis for $(TM, g)$.  Furthermore we define the \emph{Bismut scalar curvature} as
\begin{align*}
s^B = \sum_{i=1}^{2m} \rho^B(e_i, J e_i).
\end{align*}
\end{defn}
Using  the  formulas (\ref{Bismutformula}) and (\ref{Chernformula}), 
one can check (see  e.g.~\cite[Rem.~5]{gauduchon-connection})  that the induced unitary connections $\nabla^B$ and $\nabla^C$ on the anti-canonical line bundle  $K^{-1}(M,J)= \gL^m(T^{1,0}M)$, with respect to the induced Hermitian metric by $g$,  satisfy 
\begin{equation}\label{paul}
\nabla^B_X = \nabla^C_X  +  i \theta_J(JX),
\end{equation}
where 
\begin{equation*}
  \theta_J:= \IP{ \omega_J , \d\omega_J} = - \omega_J^{-1} \hook \d\omega_J
  \end{equation*}
 is the {\it Lee form} of $(g, J)$.  It follows from \eqref{paul} that the corresponding Ricci curvatures are related by
\begin{equation}\label{ricci-relation}
\rho^B = \rho^C + \d  (J^{\sharp} \theta_J),
\end{equation}
where for a $1$-form we have set $(J^{\sharp} \alpha)(X) := -(J^*\alpha)(X)= -\alpha(JX)$ for the induced action of $J$ on $T^*M$,  compatible with the $g$-duality between $TM$ and $T^*M$. Furthermore, as the $\nabla^C$ induces the Cauchy-Riemann operator $\bar \partial$ on the canonical bundle $\Omega^{(m,0)}(M, J)$, the Chern-Ricci form can be computed from any (local) holomorphic section $\Theta$ of  $\Omega^{(m,0)}(M, J)$,  via the formula 
\begin{equation}\label{chern-form}
\rho^C = -\tfrac{1}{2} \d\d^c_J \log\left(\frac{\omega_J^{[m]}}{v_{\Theta}}\right),
\end{equation}
where $v_{\Theta}:= (\sqrt{-1})^m(\Theta \wedge \bar \Theta)$.  Lastly we record an identity relating the Bismut-Ricci form and the Riemannian Ricci curvature.  This follows from (\cite{IP}, Proposition 3.1), taking in mind the sign difference for the torsion $H$.
\begin{equation}\label{Bismut-Ricci-pluriclosed}
\begin{split}
 \rho^B (X, JY)  =&\ {\rm Ric}^g(X,Y) - \tfrac{1}{2}\langle (\imath_X H), (\imath_{Y} H) \rangle_g + \tfrac{1}{2}({\mathcal L}_{\theta^{\sharp}} g)(X,Y) \\
                           &\ -  \tfrac{1}{2}(\delta^g H)(X,Y)  + \tfrac{1}{2} (\d\theta)(X, Y) -  \tfrac{1}{2} (\imath_{\theta^{\sharp}} H)_{X,Y},
\end{split}
\end{equation}

\subsection{Generalized K\"ahler structures of symplectic type} 

\begin{defn} Given a smooth manifold $M$, a \emph{generalized K\"ahler structure} is a triple $(g, I, J)$ consisting of a Riemannian metric $g$, and two integrable complex structures $I$ and $J$ compatible with $g$ such that
\begin{align*}
\d^c_I \gw_I = H = - \d^c_J \gw_J, \qquad \d H = 0.
\end{align*}
\end{defn}

\begin{defn} \label{d:sympGK} Given a smooth manifold $M$, a generalized K\"ahler structure $(g, I, J)$ is of \emph{symplectic type} if
\begin{align} \label{f:symptype}
\det (I + J) \neq 0
\end{align}
everywhere on $M$.  In this setting we define
\begin{align*}
F = - 2 g (J + I)^{-1}, \qquad b = - g (J + I)^{-1}(I - J), \qquad F^{\sharp} = F g^{-1} = - 2 (J + I)^{-1}.
\end{align*}
It follows \cite{Gualtieri-Poisson} that $F$ is non-degenerate and tames both $J$ and $I$, and $b$ is a skew-symmetric tensor which further satisfies $b \in \Wedge^{2,0 + 0,2}_I \cap \Wedge^{2,0 + 0,2}_J$.  Furthermore, direct computations yield the basic relationship
\begin{align}\label{basic}
- F J = g + b, \qquad -FI =g - b.
\end{align}
This decomposition yields the further useful inequality 
\begin{align}\label{volume}
F^m \ge \omega_J^m= m! dV_g.
\end{align}
Several differential operators are relevant for these structures.  In particular, we set
\begin{align*}
d^c_I = I d I, \qquad d^c_J = J d J, \qquad d_F = F^{\sharp} d F^{\sharp}.
\end{align*}
Using the integrability conditions for generalized K\"ahler structures it furthermore follows that (cf. \cite[Lemma~2.7]{AS})
\begin{align*}
\d F = 0, \qquad H =  \d b.
\end{align*}
\end{defn}

\begin{rmk} It turns out (cf. \cite{EFG} Proposition 1) that the data above is equivalently described by fixing first a symplectic form $F$, then asking for an integrable complex structure $J$ which is tamed by $F$, and furthermore the almost complex structure
\begin{align*}
I = - F^{-1} J^* F
\end{align*}
is integrable.  Given these conditions, the formulas in Definition \ref{d:sympGK} can be inverted to derive $g$ and $b$.  This in turn allows for an explicit derivation of the associated \emph{generalized complex structures} (cf. \cite{Gualtieri-PhD, Gualtieri-Ann}), one of which is of `symplectic type,' determined explicitly by the global closed pure spinor $F$.
\end{rmk}

K\"ahler metrics are naturally interpreted as symplectic-type generalized K\"ahler, setting $I = J$.  Furthermore, given a hyperK\"ahler structure $(g,I,J,K)$, the triple $(g,I,J)$ is GK of symplectic type, in fact is of \emph{nondegenerate} type (cf. \cite{SNDG}, \cite{AS}).  Hamiltonian deformations of these structures are examples of non-K\"ahler symplectic-type GK structures (\cite{AGG}), as are Hitchin's Poisson deformations of K\"ahler metrics on Fano surfaces \cite{HitchinJSG}.  We further record here a basic further identity:

\begin{lemma}\label{symplectic-type} Let $(g, I, J)$ be a generalized K\"ahler structure of symplectic type and denote by $\theta_J= J\delta^g \omega_J$ and $\theta_I := I\delta^g \omega_I$ the Lee forms. Then we have
\[  \frac{1}{2}d \log \det (I+J)= - (b \righthalfcup d b) = \theta_J + \imath_{\theta_J^{\sharp}} b + \delta^g b =  \theta_I  -\imath_{\theta_I^{\sharp}} b  - \delta^g b.   \]
\begin{proof} This follows directly from \cite{AS} Lemmas 2.10 and 2.13.
\end{proof}
\end{lemma}

\subsection{The Bismut-Ricci forms for symplectic-type GK structures}

Observe that in the context of generalized K\"ahler geometry we have two Bismut connections,  defined via
\begin{align*}
g(\N^{B,I}_X Y, Z)=&\ g(\N_X Y, Z) - \tfrac{1}{2} \d^c_I \gw_I (X,Y,Z) =
g(\N_X Y, Z) - \tfrac{1}{2} H(X,Y,Z),\\
g(\N^{B,J}_X Y, Z) =&\  g(\N_X Y, Z) - \tfrac{1}{2} \d^c_J \gw_J (X,Y,Z) =
g(\N_X Y, Z) + \tfrac{1}{2} H(X,Y,Z).
\end{align*}
For a generalized K\"ahler structure $(g, I, J)$, we denote  by $\rho^C_J, \rho^C_I$ and $\rho^B_J, \rho^B_I$ the Chern and Bismut-Ricci forms of the Hermitian structure $(g, J)$ and $(g, I)$, respectively.  We next derive a relationship between the two associated Bismut-Ricci forms for GK structures of symplectic type.

\begin{prop}\label{p:Ricci-difference}
	Let $(g,I,J)$ be a generalized K\"ahler structure of symplectic type. Then
	\begin{equation*}
	\rho_J^B(X,Y)-\rho_I^B(X,Y)=\tfrac{1}{2}(\mathcal L_{\theta_J^\sharp-\theta_I^\sharp}F)(X,Y).
	\end{equation*}
\end{prop}
\begin{proof}
An important identity~\cite[Lemma 9.27]{GRFBook} reads as
\begin{equation}\label{GRF-Book}
g((\mathcal{L}_{\theta_I^\sharp-\theta_J^\sharp}I) X,Y)=\rho_J^B(X,[J,I]Y),
\end{equation}
where the upper script $\sharp$ stands for the bundle map $g^{-1} : T^*M \to TM$. 
	Using this together with the analogous identity involving $\mathcal{L}_{\theta_I^\sharp-\theta_J^\sharp}J$ we have
	\[
	\begin{split}
		g((\mathcal{L}_{\theta_I^\sharp-\theta_J^\sharp}I) X,Y)=\rho_J^B(X,[J,I]Y), \qquad
		g((\mathcal{L}_{\theta_I^\sharp-\theta_J^\sharp}J) X,Y)=\rho_I^B(X,[J,I]Y).
	\end{split}
	\]
	Next, using \eqref{Bismut-Ricci-pluriclosed}, we find
	\[
	\left(\mathcal{L}_{\theta_I^\sharp-\theta_J^\sharp}g\right)(X,Y)=\rho_I^B(X,IY)+\rho_I^B(Y,IX)-\rho_J^B(X,JY)-\rho_J^B(Y,JX).
	\]
	Combining the above for $F=-2g(I+J)^{-1}$, we find:
	\[
	\begin{split}
	\tfrac{1}{2}(&\mathcal L_{\theta_I^\sharp-\theta_J^\sharp}F)(X,Y) \\
	=&\ -\mathcal L_{\theta_I^\sharp-\theta_J^\sharp} g((I+J)^{-1}X,Y)+ g((I+J)^{-1}\mathcal L_{\theta_I^\sharp-\theta_J^\sharp}(I+J)(I+J)^{-1}X,Y)\\
	=&\ \rho_J^B((I+J)^{-1}X,JY)+\rho_J^B(Y,J(I+J)^{-1}X) - \rho_I^B((I+J)^{-1}X,IY)-\rho_J^B(Y,I(I+J)^{-1}X)\\
	&\ \qquad -\rho_J^B((I+J)^{-1}X,[J,I](I+J)^{-1}Y)-\rho_I^B((I+J)^{-1}X,[J,I](I+J)^{-1}Y)\\
	=&\ \rho_J^B((I+J)^{-1}X,IY)+\rho_J^B(Y,J(I+J)^{-1}X)-\rho_I^B((I+J)^{-1}X,JY)-\rho_J^B(Y,I(I+J)^{-1}X)\\
	=&\ -\rho_J^B(X,Y)+\rho_I^B(X,Y),
	\end{split}
	\]
	where we used that $[J,I](I+J)^{-1}=(J-I)$ and that $\rho_J^B$ (resp.\,$\rho_I^B$) is $I$ invariant (resp.\,$J$-invariant), see~\cite[Lemma 9.26]{GRFBook}.
\end{proof}

\begin{prop}\label{p:symplectic-Ricci-potential} Let $(g, I, J)$ be a generalized K\"ahler structure of symplectic type and $v_{\Theta}= (\sqrt{-1})^m\Theta \wedge \bar \Theta$ a local holomorphic volume form of $(M, J)$, associated to a local non-vanishing section $\Theta$ of the canonical line bundle $\Omega^{m, 0}(M,J)$. Then the smooth function
\begin{align*}
\Phi_{J, \Theta}:= -\frac{1}{2}\left[ \log\left(\frac{F^{[m]}}{v_{\Theta}}\right) + \log\left(\frac{F^{[m]}}{\omega_J^{[m]}}\right)\right]
\end{align*}
is a potential for $\rho^B_I$, i.e.
$\rho^B_I= \d\d^c_J \Phi_{J, \Theta}$.
\end{prop}
\begin{proof}  Using \eqref{chern-form} and $F^{[m]}/\omega_J^{[m]}= (2^m) \det(I+J)^{-\frac{1}{2}}$, we get 
\[
\begin{split}
\d \d^c_J \Phi_{J, \Theta} =&\ \rho^C_J  + \tfrac{1}{2} \d \d_J^c \log \det (I+J)
                                    = \rho^C_J + \tfrac{1}{2} \d\left(J(\theta_J + \theta_I) + J b (\theta_J^{\sharp} - \theta_I^{\sharp})\right) \\
                                    =&\ \rho^C_J + \d (J \theta_J) +  \tfrac{1}{2} \d \big(F(\theta_I^{\sharp}-\theta_J^{\sharp})\big)
                                    =\ \rho^B_J  + \tfrac{1}{2} \cL_{\theta_{I}^{\sharp} - \theta_J^{\sharp}} F
                                    =\ \rho^B_I, 
\end{split}
\]
where for passing from the first to the second line we have used Lemma~\ref{symplectic-type}, for passing from the second to the third line we have used \eqref{basic},  and  for passing from the third to the fourth line we have used  Proposition~\ref{p:Ricci-difference} and  \eqref{ricci-relation}.
\end{proof}

\subsection{The generalized K\"ahler-Ricci flow} 

The generalized K\"ahler-Ricci flow (GKRF) is a solution $(g_t, I_t,J)$ of the equation
\begin{equation}\label{e:GKRF}
\begin{split}
\frac{\partial}{\partial t} \omega_J  =&\ - 2\Big(\rho^{B}_J\Big)_J^{1,1},\qquad
\frac{\del}{\del t} I  = \ \cL_{\theta_I^{\sharp} - \theta_J^{\sharp}} I.
\end{split}
\end{equation}
We refer the reader to \cite{GRFBook, ST1, ST2, STGK} for the basic properties of the above equation.  The expression given here is the equivalent expresion for GKRF in the ``$J$-fixed gauge,'' (cf. \cite{GRFBook} Definition 9.22).  Our first goal is to show that the condition of being symplectic type is preserved, which follows by recasting the GKRF in terms of $F$ and $b$.

\begin{lemma}\label{evolution-b} Let $(g_t, I_t, J)$ denote a solution to generalized K\"ahler-Ricci flow \eqref{e:GKRF} which is symplectic-type for all times.  Then the associated one-parameter families of two-forms satisfy
\begin{equation}\label{b}
\frac{\part}{\part t} F_t = - 2\rho^B_J(\omega_t), \qquad
\frac{\part}{\part t} b_t = \Delta_{g_t} b_t  - \mathcal{L}_{\theta_t^{\sharp}}b_t.
\end{equation}
\end{lemma}
\begin{proof} By \eqref{Bismut-Ricci-pluriclosed},   using that $H =  \d b$ and Lemma~\ref{symplectic-type}, we obtain 
\begin{equation}\label{skew}
\begin{split}
(- 2 J \rho^B_J)^{\rm skew} =&\ -\delta^{g} H +  \d\theta - \imath_{\theta^{\sharp}}H = \Delta_{g} b   - \mathcal{L}_{\theta^{\sharp}}b.
                                                \end{split}
\end{equation}
As $\pi_{\Wedge^{1,1}_J} F = \gw_J$ and $(-FJ)^{\rm skew} = b$, equation \eqref{skew} shows that the two equalities in the lemma are equivalent.  Using the identity \eqref{GRF-Book} and equations~\eqref{e:GKRF}
we compute below the evolution of $F=-2g(I+J)^{-1}$:
\begin{equation*}
\begin{split}
& \frac{\del}{\del t} F(X,Y)
=\ 
2\left(\frac{\del}{\del t}g\right)(X,(I+J)^{-1}Y)
+2g\left(X,\frac{\del}{\del t}(I+J)^{-1}Y\right)\\
&\ =
2\left(\rho^B_J(JX,(I+J)^{-1}Y)-\rho^B_J(X,J(I+J)^{-1}Y) \right) + 2 g\left((I+J)^{-1}X,(\mathcal{L}_{\theta_I^\sharp-\theta_J^\sharp}I)(I+J)^{-1}Y\right)\\
&\ =
2\left(\rho^B_J(JX,(I+J)^{-1}Y)-\rho^B_J(X,J(I+J)^{-1}Y)\right) -2\rho^B\Bigl([J,I](I+J)^{-1}X,(I+J)^{-1}Y\Bigr)\\
&\ =
-2\left(\rho^B_J(X,J(I+J)^{-1}Y)+\rho^B_J(X,I(I+J)^{-1}Y)\right) \ = -2\rho^B_J(X,Y),
\end{split}
\end{equation*}
where for passing to the last line we have used that $\rho^B_J$ is $I$-invariant, see ~\cite[Lemma 9.26]{GRFBook}. \end{proof}

\begin{prop} Let $(g_t, I_t, J)$ denote a solution to generalized K\"ahler-Ricci flow such that $(g_0, I_0, J)$ is symplectic-type.  Then $(g_t, I_t, J)$ is symplectic-type for all $t$ for which the flow is defined.
\begin{proof} If the maximal interval on which the data is symplectic type is $[0,T), T < \infty$, but the solution exists smoothly on $[0,T]$.  By integrating the evolution equations over $[0,T]$ it follows that the family of endomorphisms $J + I_t$ extends smoothly across time $T$.  Using the evolution equation for $F$ from Lemma \ref{evolution-b} and integrating in time we also conclude that
\begin{align*}
(J + I_t)^{-1} = - \tfrac{1}{2} g_t^{-1} F_t
\end{align*}
extends smoothly across time $T$.  It follows directly that $J + I_T$ is well-defined and invertible, contradicting maximality of $T$.
\end{proof}
\end{prop}

\subsection{The normalized generalized K\"ahler Ricci flow in the Fano case}

As explained in the introduction, in the special case when the symplectic form $F$ of a symplectic type generalized K\"ahler structure $(g, I, J)$ belongs to $2\pi c_1(M, J)$, we can consider the \emph{normalized} GKRF
\begin{equation}\label{e:NGKRF-2}
\begin{split}
\frac{\partial}{\partial t} \omega_J  =&\ - 2\Big(\rho^{B}_J\Big)_J^{1,1} + 2\omega_J, \qquad
\frac{\del}{\del t} I  =\ \cL_{\theta_I^{\sharp} - \theta_J^{\sharp}} I.
\end{split}
\end{equation}
To get a sense of the possible global existence behaviour of this in the Fano case, we exhibit the behavior of some basic quantities.

\begin{prop} \label{p:NGKRFPoisson} Let $(M^{2n}, J)$ be a Fano manifold, and suppose $(g_t, I_t, J)$ is a solution of (\ref{e:NGKRF-2}) of symplectic type such that $[F_0] = 2 \pi c_1(M, J)$.  Then for all $t$, the cohomology class $[F_t]$ and the Poisson tensor $\gs_t=[I_t,J] g_t^{-1}$ (see~\ref{GK-Poisson}) satisfy
\begin{enumerate}
\item $[F_t] = 2 \pi c_1(M, J)$,
\item $\gs_t = e^{-2t} \gs_0$.
\end{enumerate}
\begin{proof} The solutions $(g_t, I_t)$ and $(\tilde{g}_t, \tI_t)$ of \eqref{e:GKRF} and \eqref{e:NGKRF-2} are related by a reparametrisation
\begin{align*}
\tilde{g}_t: = \tfrac{1}{2}e^{2t} g_{(1-e^{-2t})}, \qquad \tI_t=I_{(1-e^{2t})}.
\end{align*}
It follows directly from Lemma \ref{evolution-b} that
\begin{equation*}
\frac{\partial}{\partial t} \tilde{F} = -2\rho^B_J(\tilde{g}) + 2 \tilde{F},
\end{equation*}
proving claim (1).  To see claim (2), we first note that along a solution of (\ref{e:GKRF}) one has that $\gs$ is fixed in time (\cite[Corollary 1.5]{GS} cf. also ~\cite[Proposition 9.29]{GRFBook}).  Claim (2) then follows easily from the reparameterization relationships above.
\end{proof}
\end{prop}

Thus the symplectic form $F$ remains in the canonical class, while the tensor $\sigma$, which measures the deviation of a symplectic-type GK structure from being K\"ahler, decays exponentially.  Based on this, as well as the monotonicity formulae described below in \S \ref{s:MF}, one might expect the normalized GKRF \eqref{e:NGKRF-2} to behave similarly to the much studied normalized K\"ahler-Ricci flow on Fano manifolds. In particular, by a result of Cao~\cite{Cao} the solution of the normalized  K\"ahler-Ricci flow exists for all time, and, by results of Tian-Zhu~\cite{TZ-07, TZ-13}, the global solution of the normalized K\"ahler-Ricci flow composed by automorphisms of $(M, J)$ converges to a K\"ahler-Ricci soliton, provided that the latter exists. Furthermore, by \cite{Zhu}, in the toric case the reduced flow will converge at the level of normalized potentials.  It is natural to ask whether or not similar results hold true for solutions of \eqref{e:NGKRF-2}.  More precisely we have:

\begin{conj}\label{main-conjecture} Suppose $(g, I, J)$ is a symplectic type generalized K\"ahler structure on a Fano manifold $(M, J)$, with $F \in 2\pi c_1(M, J)$. Then
\begin{itemize}
\item The solution of \eqref{e:NGKRF-2} with this initial data exists for all time $t\in [0, \infty)$.
\item If $(M, J)$ admits a K\"ahler-Einstein metric then the global solution of \eqref{e:NGKRF-2} converges, in the $C^{\infty}(M)$ topology, to a K\"ahler-Einstein metric.
\item If $(M, J)$ admits a K\"ahler-Ricci soliton, with a soliton  vector field $K$, and $(g, I, J)$ is invariant by the torus action generated by $K$,  then there exists a smooth family of complex automorphisms $\tau_t$ of $(M, J)$ such that the global solution $g_t$ of \eqref{e:NGKRF-2} pulled-back by $\tau_t$ converges, in the $C^{\infty}(M)$ topology,  to a K\"ahler-Ricci soliton metric in $c_1(M, J)$ with soliton vector field $K$.
\item If $(M, J)$ is a toric Fano variety and $(g, I, J)$ is invariant under the maximal torus  $\T$, then  the reduced equation \eqref{toric-GKRF-Kahler} has a global solution $\phi_t(y)$ defined on $[0, \infty)\times \R^m$, and there are families of points $y_t\in \R^n$ and real constants $c_t$,  such that $\tilde \phi_t (y):= \phi_t(y+ y_t)+ c_t$ converges in $C^{\infty}([0, \infty)\times \R^m)$ to a smooth convex  function $\tilde \phi_{\infty}(y)$ on $\R^m$ which defines a $\T$-invariant K\"ahler Ricci soliton on $M$.
\end{itemize}
\end{conj}

\section{Toric generalized K\"ahler structures}\label{s:toric-GK}
\subsection{Toric K\"ahler structures}\label{s:toric-Kahler} 

In this section we recall fundamental material on toric K\"ahler structures going back to~\cite{Abreu2, guillemin};  we refer to~\cite{apostolov-notes,donaldson-survey} for comprehensive surveys.  Let $(M, \Fo)$ be a smooth compact symplectic manifold of real dimension $2m$, endowed with an effective Hamiltonian action of a compact $m$-dimensional torus $\T$. We denote by $\tor$ the Lie algebra of $\T$, and  by $\tor^*$ the dual vector space. Let $\mu  : M \to \tor^*$  be a fixed momentum map for the action.  Delzant's theorem~\cite{delzant}   tells us that  the  image of $\mu(M)$ is a Delzant polytope $\Pol \subset \tor^*$, which comes equipped with a minimal set of defining hyperplanes 
\[{\Pol }= \{ x \in \tor^* : L_j(x) =\langle v_j, x\rangle + \lambda_j \ge 0, \, j=1, \ldots, d\},  \]
where $v_j$ are primitive elements of the lattice $\Lambda \subset \tor$ of circle subgroups of $\T$, i.e. $\T= \tor/2\pi \Lambda$.  We denote by  ${\La} =\{L_1(x), \ldots, L_d(x)\}$ the set of the defining affine-linear functions of $\Pol$ as above, which sometimes is referred to as a {\it labelling} of $\Pol$, see \cite{LT}.  According to \cite{delzant},  the data $(\Pol, {\La})$  in turn identify  $(M, \Fo)$,  up to a $\T$-equivariant symplectomorphism,  with   the K\"ahler reduction of $\C^d$ (endowed with its flat K\"ahler structure) by  a real torus of dimension $d-m$. This gives rise to  a {\it canonical}  $\Fo$-compatible $\T$-invariant K\"ahler structure $(g_c,  J_c, \Fo)$ on $M$.
 
We next describe  the set  $\mathcal{K}^{\T}(M,\Fo)$ of $\T$-invariant, $\Fo$-compatible  complex structures on $(M, \Fo)$, or equivalently, the set of  $\T$-invariant K\"ahler metrics  $(g, J, \Fo)$ with fixed K\"ahler form $\Fo$.   Notice that $J_c\in \mathcal{K}^{\T}(M, \Fo)$ by construction.  On the union $M^0:=\mu^{-1}(\mathring{\Pol})$ of the generic orbits of the $\T$-action (where $\mathring{\Pol}$ denotes the interior of $\Pol$), any K\"ahler metric $g$  determined by an element $J$ in $\mathcal{K}^{\T}(M, \Fo)$
has  a general expression due to V.~Guillemin~\cite{guillemin}. In this description, the
momentum map $\mu :  M^0\to \mathfrak{t}^*$ is supplemented by angular coordinates (depending on $J$)
$\ang : M^0 \to \tor/2\pi\Lambda$,  such that the kernel of $\d\ang$ is
orthogonal to the tangent space of the torus orbits. These momentum-angular coordinates $(\mu,\ang)$
identify each tangent space to $M^0$ with $\tor\oplus \tor^*$, and the
symplectic form $\Fo$  and a compatible K\"ahler metric $g$ on $M^0$  have the form
\begin{equation}\label{toricmetric}
\Fo=\langle \d\mu, \d\ang\rangle, \qquad g =\ip{\d\mu, {\mathbf G}(\mu) , \d\mu}+ \ip{ \d\ang,{\mathbf H}(\mu), \d\ang}, 
\end{equation}
where $\ip{\cdot, \cdot}$
denotes contraction of $\tor$ and $\tor^*$, ${\mathbf G}(\mu)$ is  (the pull back by $\mu$) of a positive definite $S^2\tor$-valued function on $\mathring{\Pol}$
(with $S^2\tor$ denoting  the symmetric tensor product of $\tor$),
${\mathbf H}$ is its point-wise inverse in $S^2\tor^*$, and $\ip{\cdot,\cdot,\cdot}$ denotes the
point-wise contraction $\tor^* \times S^2\tor \times \tor^* \to \R$ or the
dual contraction.  The corresponding complex structure $J$  is defined by
\begin{align}\label{toricJ}
J \d\ang = -\ip{ {\mathbf G}(\mu), \d\mu}, \qquad  J \d \mu = \ip{ {\mathbf H}(\mu),  \d\ang}, 
\end{align}
and $J$ is integrable if and only if ${\mathbf G}= {\rm Hess}(u)$ is the Hessian of a
smooth strictly convex function $u(x)$ on $\mathring{\Pol}$, called \emph{symplectic potential} of $(g, J, \Fo)$ (cf.~\cite{Abreu2, guillemin}). This applies in particular to the canonical complex structure $J_c\in \mathcal{K}^{\T}(M,\Fo)$ coming from the Delzant construction, giving rise to canonical angular coordinates $\ang^c$ and symplectic potential~\cite{guillemin}
\begin{equation}\label{u-canonical}
u_c(x) := \frac{1}{2}\sum_{j=1}^d L_j (x)\log L_j(x).
\end{equation}
Necessary and sufficient conditions for the  symplectic potential $u$ to come from a globally defined  K\"ahler metric  $(g, J)$ on $M$ are obtained in \cite{Abreu2,ACGT2,donaldson}.  Here we
mention the conditions established in \cite{donaldson}: 
\begin{prop}\label{p:boundary} A smooth strictly convex function $u(x)$ on $\mathring{\Pol}$ is a symplectic potential of a globally defined $\Fo$-compatible K\"ahler metric in $\mathcal{K}^{\T}(M, \Fo)$ if and only if $u$ satisfies the following boundary conditions:
\begin{bulletlist}
\item  $u(x)$ is smooth and strictly convex in the interior of $\Pol$,  and extends as a continuous function  on $\Pol$ which is smooth and strictly convex on the interior of each face of $\Pol$;
\item  $u(x)-u_c(x)$ extends to a smooth function over $\Pol$, where $u_c(x)$ is given by \eqref{u-canonical}.
\end{bulletlist}
\end{prop}

\begin{defn}\label{d:symplectic-potential}
We denote by $\mathcal{S}(\Pol, {\La})$ the space of functions $u$ satisfying the conditions of Proposition~\ref{p:boundary}.
\end{defn}

A subtle point in the above description is the fact that the angular coordinates $\ang$ depend upon $J$, but one can show~\cite[Lemma~3]{ACGT2} that by pulling back $J$ by a $\T$-equivariant symplectomorphism of $(M,\Fo)$, we can assume that $\ang=\ang^c$, a  normalization we shall implicitly use below without further notice.
It follows that the space $\mathcal{K}^{\T}(M, \Fo)$ modulo  pull-backs by  $\T$-equivariant  symplectomorphisms can be identified with the space $\mathcal{S}(\Pol, \La)$,  modulo  additions with affine-linear functions.  Because of this correspondence, for any  $u\in \mathcal{S}(\Pol, \La)$, with a  slight abuse of notation, we denote by $(g_u, J_u)$ the corresponding $\T$-invariant K\"ahler metric on $M$ (where we implicitly use the canonical angular coordinates $\ang=\ang^c$ to define $(g_u, J_u)$).

It is a basic fact of the theory (see e.g. \cite{LT}) that any two elements of $\mathcal{K}^{\T}(M, \Fo)$ are biholomorphic under a $\T$-equivariant diffeomorphism,   which acts trivially on the cohomology class of  $[\Fo]$, i.e.  for any $J \in\mathcal{K}^{\T}(M, \Fo)$, there exists a $\T$-equivariant diffeomorphism $\Phi$,  such that $\Phi\cdot J = J_c$ and  $\Phi^*\Fo$  is the  K\"ahler form  of a $\T$-invariant K\"ahler metric in the K\"ahler class $[\Fo]$ on the  fixed complex manifold $(M,J_c)$. The converse is also true by the equivariant Moser lemma. These two equivalent descriptions are often referred to as {\it symplectic}  versus {\it complex} point of view, respectively, and will be used in our study below. We  thus give next  an explicit description of this correspondence, based on  an observation from \cite{guillemin}.
 
For any $u\in \mathcal{S}(\Pol, {\La})$, we let 
\begin{align}\label{Legendre}
y= \nabla u, \qquad  \phi(y) + u(x) =\langle y, x\rangle
\end{align}
be the Legendre transform of the strictly convex smooth function $u$ on $\mathring{\Pol}$. Using the compactness of $\mathring{\Pol}$ and the strict convexity of $u$, one can see that 
\[ y=\nabla u : \mathring{\Pol} \to  \tor \cong \R^m \]
is a diffeomorphism whereas $\phi(y)$ is a strictly convex smooth function on $\R^m$.   We denote  by $\mathcal{C}({\R}^m)$ the space of such functions on $\R^m$  and let 
\[
{\mathcal T} : \mathcal{S}({\Pol}, {\La})  \mapsto \mathcal{C}({\R}^m),  \qquad  \mathcal{T} (u(x))  = \phi(y)
\]
be the induced map  on the corresponding Frech\'et spaces.  Notice that  we can recover  $u(x)$  from  its image  $\phi(y)$ by letting $x=(\nabla \phi)(y)$
and using \eqref{Legendre}.   A simple computation using the definition \eqref{Legendre} shows that the differential of $\mathcal{T}$ is given by
\begin{align}\label{Legendre-differential}
(\delta_{u} \mathcal{T})(\dot{u}) = -\dot{\phi}.
\end{align}

It turns out that the Legendre transform $\mathcal{T}$ underlines the geometric correspondence between the symplectic and complex points of view mentioned above. To see this,
for any $J=J_u \in \mathcal{K}^{\T}(M, \Fo)$, we shall  implicitly identify the abstract variable $y=\nabla u \in \tor$ with the function $y : M \to \tor$ given by  the composition $y=\nabla u(\mu)$, and use  $\{y(\mu), \ang^c\}$ as a system of coordinates on $M^0$.  Let us  introduce a basis $\{K_j = \frac{\partial}{\partial \ang_j} \}$ for $\tor$ (usually given by the generators of circle groups of $\T$) and write $\ang=(\ang_1, \ldots, \ang_m)$ and $\mu=(\mu_1, \ldots, \mu_m)$ with respect to this and the dual basis. It  is easily seen that the $1$-forms $\{\d y_1, \ldots, \d y_m, \d \ang_1, \ldots, \d \ang_m\}$ form a dual  basis (at each point of $M^0$) of the basis of commuting real holomorphic vector fields $\{-JK_1, \ldots, -JK_m, K_1, \ldots, K_m\}$.   We can thus consider  two systems of holomorphic  coordinates on $(M^0, J)$: $ \{ y_1 + \sqrt{-1} \ang_1, \ldots, y_m+ \sqrt{-1} \ang_m\}$ and 
 $\{z_1, \ldots, z_m\}$ with $z_j:= e^{y_j + \sqrt{-1} \ang_j}$. Geometrically, the action of $\T$  gives rise to  a  holomorphic action  of the complex torus $\T_{\C} \cong (\C^{*})^m$ on $(M, J)$,  which  is the linear in the coordinates $\{z_1, \ldots, z_m\}$ on $M^0$.  In other words, letting $x_u\in \Pol^0$ be the pre-image of $0\in \tor$ under $\nabla u$ (or, equivalently, the  point of minima of $u$)  and $p_u$  the point in $M^0$ such that $\mu(p_u)= x_{u}$ and $\ang(p_u)=0$,  we have the  identification $(M^0, J)=(\C^*)^m \cdot p_u$ with the orbit of $\T_{\C}$  at $p_u$,   with  $z_j= e^{y_j + \sqrt{-1}\ang_j}$ being the natural complex coordinate on the $j$-th factor $\C^*$   (reflecting the fact that  $z_j(p_u) =1$).  The key observation in \cite{guillemin} is that the K\"ahler form $\Fo$  is written on $(M^0, J)$ as
 \begin{equation}\label{Kahler-toric-potential}
\Fo = \d \d^c_J \phi(y(\mu)) = \sum_{i,j=1}^m \phi_{, pq}(y) \d y_p \wedge \d \ang_q,
 \end{equation}
where $\phi=\mathcal{T}(u)$ and, by Legendre duality,  $\big(\phi_{,pq}(y)\big)  = \big(u_{ij}(x)\big)^{-1} = {\bf G}^{-1}(x)={\bf H}(x)$. 

The above description holds on $M^0$ in complex coordinates  with respect to  $J=J_{u}$. To  relate it to the canonical complex structure $J_c$ of $M$, let  $z^c=(z_1^c, \ldots, z_m^c)$ be  the respective holomorphic  coordinates on $(M^0, J_c)$,  obtained  from the  canonical potential $u_c(x)$ defined by \eqref{u-canonical}. We let $p_0:=p_{u_0}$ be the corresponding  point in $M^0$ so that $(M^0, J_c) = (\C^*)^m \cdot p_0$  with   $(z_1^c, \ldots, z_m^c)$  being the standard coordinates on $(\C^*)^m$. Then,  using the boundary conditions for  $u$ in Proposition~\ref{p:boundary},  the following correspondence is established  in ~\cite[pp. 395-396]{Do-08}:
\begin{prop}\label{p:symplectic-to-complex} The Legendre transform $\mathcal{T}$  defines a bijective map from the space of symplectic potentials   ${\mathcal S}(\Pol, {\La})$   to the space $\mathcal{K}(\R^m, \Fo)$ of strictly convex smooth functions $\phi(y)$ on $\R^m$,    satisfying the following conditions:
\begin{bulletlist}
\item the function defined on $(\C^*)^m\cdot p_0 =  (M^0, J_c)$ by  \[\phi(z):= \phi\big(\tfrac{1}{2}\log|z^c_1|^2, \ldots, \tfrac{1}{2}\log |z^c_m|^2\big)\]  gives rise to  a K\"ahler metric $\omega_{\phi} = \d \d^c \phi$  on $(M, J_c)$.
\item $\phi(z^c)-\phi_c(z^c)$  extends to a smooth function on $(M, J_c)$.
\end{bulletlist} 
In particular, $\mathcal{T}$ gives rise to a bijection between the space of $\T$-invariant $\Fo$-compatible K\"ahler metrics  on $M$,  modulo the action of $\T$-equivariant symplectomorphisms of $(M, \Fo)$, and the space of $\T$-invariant K\"ahler metrics in the K\"ahler class $[\Fo]$ on $(M, J_c)$, modulo the action of $\T_{\C}=(\C^*)^m$. 
\end{prop}

\subsection{Toric generalized K\"ahler structures}
Inspired by the theory of toric K\"ahler structures, L. Boulanger~\cite{boulanger} and Y.~Wang~\cite{W1,W2}  gave a similar description in momentum-angular coordinates of 
the $\T$-invariant generalized  K\"ahler structures compatible with  a symplectic $2$-form $\Fi$  on a smooth compact toric symplectic manifold $(M, \Fi)$. In the notation  of Section~\ref{s:toric-Kahler},   we let $(M, \Fo, \T)$ be a given toric symplectic manifold with momentum map $\mu: M \to \tor^*$ and $\ang: M^0 \to \tor/2\pi \Lambda$  standard   angular coordinates on $M^0$ (i.e. $\ang=\ang^c$);  choosing a lattice basis of $\tor$,  we write (see \eqref{toricmetric}) 
\begin{equation}\label{omega}
\Fo= \sum_{j=1}^m \d \mu_j \wedge \d \ang_j.\end{equation}

\subsubsection{Toric generalized K\"ahler deformations of type $A$}

Following  \cite{boulanger}, we can construct  a $\T$-invariant  almost complex structure $J$  on $M^0$  from a smooth  non-degenerate field ${\bf \Psi}(x)$ of bilinear forms on $\mathring{\Pol}$, with inverse ${\bf \Psi}^{-1}(x)$, by letting
\begin{equation}\label{boulanger}
J:= \sum_{i,j=1}^m \Big( {\bf \Psi}_{ij}(\mu) \d\mu_i \otimes \frac{\partial}{\partial \ang_j}- ({\bf \Psi}^{-1})_{ij}(\mu)\d\ang_i \otimes  \frac{\partial}{\partial \mu_j}\Big), 
\end{equation}
where $({\bf \Psi}_{ij}(x))$ is the Gram matrix (in the  fixed basis of $\tor$ and $\tor^*$) of  ${\bf \Psi}(x)$.  This is consistent with the description \eqref{toricJ} of toric K\.ahler structures (which correspond to the case  ${\bf \Psi}(x) = {\bf G}(x)= {\rm Hess}(u(x))$).

The integrability of the almost complex structure  $J$  given by \eqref{boulanger} reads as (see  the proof of \cite[Theorem~6]{boulanger})
\begin{equation*}
{\bf \Psi}_{ij, k} ={\bf  \Psi}_{ik,j}.
\end{equation*}
The almost-complex structure $I= -\Fo^{-1} J^*\Fo$  is expressed in  momentum-angular coordinates by 
\begin{equation*}
I = \sum_{i,j=1}^m \Big( ({\bf \Psi}^{\rm T})_{ij} \d\mu_i \otimes \frac{\partial}{\partial \ang_j}- \big({\bf \Psi}^{\rm T}\big)_{ij}^{-1} \d\ang_i \otimes  \frac{\partial}{\partial \mu_j}\Big).
\end{equation*}
Furthermore,  $\Fo$ tames $J$ if and only if the the symmetric  tensor
\begin{equation}\label{g}
\pg :=  -(\Fo J)^{\rm s} = \sum_{i,j=1}^m \Big( ({\bf \Psi}^{\rm s})_{ij} \d\mu_i \d\mu_j + ({\bf \Psi}^{-1})^{\rm s}_{ij} \d\ang_i\d\ang_j\Big)
\end{equation}
is positive definite on $M^0$.  In  the above formula (and in what follows) we use upper-indices ${\rm s}$ and ${\rm a}$ to denote respectively the symmetric and skew-symmetric parts  of bilinear forms and the corresponding Gram matrices.  We shall make an abundant use of  the following identities which hold  for a non-degenerate matrix ${\bf \Psi} =({\bf \Psi}_{ij})$ (they can be checked easily using the polar decomposition of ${\bf \Psi}$).
\begin{equation}\label{linear-algebra}
\begin{split}
& {\bf \Psi} \big({\bf \Psi}^{-1}\big)^{\rm s} {\bf \Psi}^{\rm T} = {\bf \Psi}^{\rm s}, \, \, {\bf \Psi} \big({\bf \Psi}^{-1}\big)^{\rm a} {\bf \Psi}^{\rm T} = -{\bf \Psi}^{\rm a}; \\
& {\bf \Psi}^{\rm s}\big({\bf \Psi}^{-1}\big)^{\rm a} + {\bf \Psi}^{\rm a} \big({\bf \Psi}^{-1}\big)^{\rm s} =0, \, \,  {\bf \Psi}^{\rm s}\big({\bf \Psi}^{-1}\big)^{\rm s} + {\bf \Psi}^{\rm a} \big({\bf \Psi}^{-1}\big)^{\rm a} ={\rm Id},
\end{split}
\end{equation}
where ${\bf \Psi}^{\rm T}$ stands for the transposed matrix.  It follows from \eqref{g} that $\pg$ is positive definite over $M^0$ if and only if  ${\bf \Psi}^{\rm s}(x)$ is  positive definite over $\mathring{\Pol}$.  

A special case when the almost complex structures $J$ and $I$  are both integrable and tamed by  $\Fo$  on $M^0$   is obtained by letting
\begin{equation}\label{toric-GK}
{\bf \Psi}(x) = {\rm Hess}(u(x)) + A
\end{equation}
where $u(x)$ is a smooth strictly convex  function on $\mathring{\Pol}$ and $A\in \gL^2\tor$ is a (constant) skew-symmetric  bilinear form on $\tor^*$, thus extending the K\"ahler setting which is obtained by letting $A=0$. 
 
Thus, \eqref{toric-GK}  gives rise to a $\T$-invariant generalized structure on $M^0$  which is compatible with the symplectic form $\Fo$  in the sense of \eqref{GK-symplectic}, as mentioned in the introduction. The following result  is established in \cite[Theorem 11]{boulanger} for $m=2$,  and in \cite[Theorems 4.9 \& 4.11]{W1} in general.

\begin{prop}\label{p:compactification} The complex structure $J_A$ defined on $M^0$ by \eqref{boulanger}  and  ${\bf \Psi}(x)$ of the form \eqref{toric-GK} for some smooth strictly convex function $u(x)$ on $\mathring{\Pol}$ and $A\in \gL^2\tor$ extends to a globally defined $\Fo$-compatible generalized K\"ahler structure $(\pg_A, I_A, J_A)$ on $M$ if and only if $u\in \mathcal{S}(\Pol, {\La})$.\end{prop}

\begin{defn}\label{d:diagonal}
The GK structures $(\pg_A, J_A, I_A)$ obtained  from a  toric $\Fo$-compatible K\"ahler structure $(g_u, J_u)$ via Proposition~\ref{p:compactification}  are called \emph{generalized K\"ahler deformations of type $A$} of $(g_u, J_u)$.
\end{defn}

For a toric  generalized K\"ahler  structure $(\pg_A, J_A, I_A)$ as above, with corresponding $u\in \mathcal{S}(\Pol, {\La})$ and  $A\in \gL^2\tor$,  we introduce pluriharmonic functions on $(M^0, J)$ by 
\begin{equation}\label{tilde-y}
\tilde y_j(\mu) :=  u_{,j}(\mu) + \sum_{k=1}^m A_{kj} \mu_k, \qquad j=1, \dots, m.
\end{equation}
Indeed, it is easily checked  from \eqref{boulanger} that,  at each point,  $\{\d\tilde y_1, \ldots, \d\tilde y_m, \d \ang_1, \ldots, \d \ang_m\}$  is  the  dual  basis of  the basis of commuting real holomorphic vector fields $\{-JK_1, \ldots, -JK_m, K_1, \ldots, K_m\}$.  
As in the case of compatible toric K\"ahler structures,  we can think of  $\tilde y$ as abstract variables on $\tor$,  defined by the transformation
\begin{equation}\label{tilde-y-inv}
{\tilde y}  := (\nabla u)(x) + A(x),
\end{equation}
where $A$  is thought of as a linear map from $\tor^* \to \tor$. Using that $u$ is strictly convex and $A$ is skew-symmetric, one can show as in the Legendre transform \eqref{Legendre} that ${\tilde y}(x)$ is a diffeomorphism from $\mathring{\Pol}$ to $\R^m$. Similarly to the construction  in the K\"ahler case, we let $\tilde x_{u}$ be the pre-image of $0\in \tor$ under \eqref{tilde-y-inv} and $\tilde p_u$ the point in $M^0$ with $\mu(\tilde p_u)=\tilde x_{u}$ and $\ang (\tilde p_u)=0$ so that $\tilde z_j = e^{\tilde y_j + \sqrt{-1} \ang_j}$  are identified with  standard holomorphic coordinates  on the  orbit  $(\C^m)\cdot \tilde p_u =(M^0, J)$.   With this understood, we have the \textit{canonical biholomorphism} between two complex structures given by deformations of type $A$.

\begin{lemma}[Canonical biholomorphism between $(M,J_A)$ and $(M,J_{A_0})$]\label{l:complex-identification} Let $J_A$ and $J_{A_0}$ be  $\T$-invariant complex structures on $M$,  obtained by deformations of type $A$ of the K\"ahler structures $(g_{u},  J_{u})$ and $(g_{u_0}, J_{u_0})$ in $\mathcal{K}^{\T}(M, \Fo)$ with corresponding symplectic potentials $u, u_0 \in \mathcal{S}(\Pol, {\La})$, respectively. Then, the $\T$-equivariant diffeomorphism $\tilde \Phi$ on $M^0$,  which sends the complex coordinates $(\tilde z_1, \ldots, \tilde z_m)$ of $J_A$ to the respective complex coordinates $(\tilde z_1^0, \ldots, \tilde z_m^0)$ of $J_{A_0}$, extends to a $(\C^*)^m$-equivariant biholomorphism  from $(M, J_A)$ to $(M, J_{A_0})$. In particular, independent of $u$ and $A$, $(M, J_A, \T)$ is isomorphic, as a complex toric variety, to $(M, J_c, \T)$, where $J_c$ is the canonical complex structure in $\mathcal{K}^{\T}(M, \Fo)$.
\end{lemma}
\begin{proof} Given some $J=J_A$ as in the statement, we shall build a holomorphic $(\C^*)^{m}$-equivariant atlas on $(M, J)$ consisting of  complex charts $\psi_v : M_v \to \C^m$ associated to each vertex $v$ of $\Pol$. To this end,  we suppose that the chosen basis  $e$ of $\tor$, defining a dual basis of $\tor^*$ and the holomorphic coordinates $(\tilde y_1, \ldots, \tilde y_m)$ via \eqref{tilde-y} on $(M^0, J)$,   is a $\Z$-basis of $\Lambda$. By Delzant theory~\cite{delzant}, each vertex $v\in \Pol$ corresponds to a point $p_v\in M$ fixed by the action of $\T$. We change the momentum map $\mu^v  := \mu-v$ by a translation,   so that  $\mu(p_v)=0\in \tor^*$. Furthermore, we change  the  initial $\Z$-basis of $\tor$ with  a $\Z$-basis $e_v$ consisting of the primitive normals of the labels $L_j\in {\La}$  vanishing at $v$,  with  corresponding  transition matrix of bases $N_v=(n_{ij}) \in {\bf GL}(m, \Z)$.  Using the new basis $e_v$  and the modified momentum map $\mu^v$,  we define via \eqref{tilde-y},  new $J$-holomorphic coordinates  $(\tilde z_1^v, \ldots, \tilde z_m^v)$ on $(M^0, J)$  by  $\tilde z^v_k= e^{(\tilde y^v_k + \sqrt{-1} \ang^v_k)}$. It follows that on $(M^0, J) \cong (\C^*)^m\cdot \tilde p_u$, the coordinates  $\tilde z$ and $\tilde z^v$ are related by
\begin{align}\label{atlas}
 \tilde z^v = e^{-\lambda_v(A) } \tilde z^{N_v}, 
 \end{align}
which is an abbreviation  to the map
\[(\tilde z^v)_k = e^{-(N_vA(v))_j}(\tilde z_1)^{n_{k1}} (\tilde z_2)^{n_{k2}} \cdots (\tilde z_m)^{n_{km}}, \, k=1, \ldots m.\]
We now show  that    $(\tilde z_1^v, \ldots, \tilde z_m^v)$ can be extended  from $(M^0, J)$  to define a $(\C^*)^m$-equivariant  holomorphic chart on $(M, J)$ centered at  $p_v$. To this end, notice that with respect to the basis $e_v$ and with  the normalization for the momentum map $\mu^v$, the labels  $L_k\in {\La}$ vanishing at $v=\mu(p_v)=0$ are $L_k(\mu^v)= \mu^v_k, k=1, \ldots m$. We will show that  $(\tilde z_1^v, \ldots, \tilde z_m^v)$  extend at any point in the pre-image $M^0_{v}$ by $\mu^v$ of the union of $v$ and the interiors of all faces of $\Pol$ containing the vertex $v$. Let us  take a point $p_*$  in the pre-image of the interior $\mathring{\Sigma}_k$ of the facet  defined by $\mu^v_k=0$. (The argument for a face of higher co-dimension  is similar.)  Letting $x_* =\mu^v(p_*)$,  by  the second boundary condition for $u$ in Proposition~\ref{p:boundary}, we have that  $u_{,k}(\mu^v)= \frac{1}{2} \log(\mu^v_k) + \varphi(\mu^v)$ for a smooth function $\varphi(x)$  defined in a  closed ball  $B_*$ around $x_*$; it thus follows that on $(\mu^v)^{-1}(B_*) \cap M^0$,  we have  $|\tilde z^v_k|^2 = (\mu^v_k) e^{\varphi(\mu^v)}$, showing that $\tilde z^v_k$ tends to zero for any sequence of points in $M^0$ which converges to  $p_*$.  The same boundary condition also ensures that  the functions $\tilde y^v_r(x)$  with $r\neq k$ are extendable as smooth functions  on  $\mathring{\Sigma}_k$.  Using the first boundary condition  of Proposition~\ref{p:boundary} and the definition \cite{delzant}  of the canonical angular coordinates $\ang_r$  (with $r\neq k$)  on $\mu_{v}^{-1}(\mathring{\Sigma}_k)$, it follows that 
$(\tilde z_1, \ldots, \tilde z_{k-1}, 0, \tilde z_{k+1}, \ldots, \tilde z_m)$  gives rise to  a $(\C^*)^m$-equivariant coordinate system on  $(\mu^v)^{-1}(\mathring{\Sigma}_k)$.  Similar arguments on any face containing $v$  yield that   $(\tilde z_1^v, \ldots, \tilde z_m^v)$  defines   a global  $(\C^*)^m$-equivariant  chart $\psi_v: M_{v} \to \C^m$.  Considering such charts for all vertices of $\Pol$ gives rise to a holomorphic atlas on $(M, J)$.

We now let $\psi^0_v : (M_v, J_0) \to \C_m$ be the corresponding chart at $p_v$  for the  complex structure $J_0=J_{A_0}$ (obtained by a deformation of type $A$ from  $u_0\in \mathcal{S}(\Pol, {\La})$ and $A_0 \in \gL^{2}\tor$).  Using \eqref{atlas}, it follows that in the charts $\psi_v, \psi_v^0$, the map $\Phi: (M^0, J)\to (M^0, J_0)$ sending $(\tilde z_1, \ldots, \tilde z_m)$ to $(\tilde z_1^0, \ldots, \tilde z_m^0)$  becomes 
\[  \tilde z^v \mapsto e^{\lambda_v(A_0-A)} (\tilde z^0)^v, \]
which is well-defined on  the whole $\C^m$. \end{proof}

\subsubsection{Toric generalized K\"ahler deformations of type $B$}\label{s:toricGK-B}  

Following \cite{W1}, let  $(\pg_A, I_A, J_A)$ be a generalized K\"ahler deformation of type $A$ on $(M, \Fo, \T)$, defined via \eqref{boulanger} and \eqref{toric-GK} by $u\in \mathcal{S}(\Pol, {\La})$ and $A\in \gL^2\tor$, and take another element $B\in \gL^2\tor$.  Then the $2$-form 
\begin{equation}\label{F-B}
\Fi := \Fo + \langle \d \mu, B, \d \mu\rangle   =   \sum_{j=1}^m \d\mu_i \wedge \d\ang_i  +  \sum_{i,j=1}^m B_{ij} \d\mu_i \wedge \d\mu_j 
         \end{equation}
is $\T$-invariant and symplectic on $M$,   with moment map  $\mu$ and associated Delzant polytope $(\Delta, {\La})$. In particular, $(M, \Fo, \T)$ and $(M, F, \T)$ are isomorphic as symplectic toric manifolds, but  the momentum-angular coordinates of $\Fi$ are $(\mu, \ang_F:=\ang - B(\mu))$.

It is observed in \cite{W1} that for any $B\in \gL^2\tor$, the almost complex structure $I_{A,B}:= -(\Fi)^{-1}(J_A)^* \Fi$  is integrable.  Indeed, with respect to the frame $(\frac{\partial}{\partial \ang_i}, \frac{\partial}{\partial \mu_j})$ on $TM^0$, $J_A$ and $I_{A,B}$ are represented by matrices
\begin{equation}\label{e:IB}
J_A \sim \left( \begin{array}{cc} 
 0 & {\bf \Psi}^{\rm T} \\
 -\big({\bf \Psi}^{\rm T}\big)^{-1} & 0  \end{array} \right), \qquad
 \, \, I _{A,B}\sim \left( \begin{array}{cc} 
 2B{\bf \Psi}^{-1} &  {\bf \Psi}  + 4B{\bf \Psi}^{-1} B \\
-{\bf \Psi}^{-1} & -2{\bf \Psi}^{-1}B \end{array} \right),
\end{equation}
 where ${\bf \Psi}$ is given by \eqref{toric-GK}. One can then check from  the above representations that 
  \begin{equation}\label{pluriharmonic-I}
 \bar y := \nabla u  - A(\mu), \qquad \bar \ang := \ang- 2 B(\mu)= \ang_F  -B(\mu)
 \end{equation}
 are pluriharmonic coordinates with respect to $I_{A,B}$  and hence  $I_{A,B}$ is integrable. 
 
Suppose now that  $J_A$  is tamed by $F$.  This  gives rise to a $\T$-invariant generalized K\"ahler structure $(\pg_{A,B}, I_{A,B}, J_{A})$,  where 
the induced Riemannian metric  $\pg_{A,B}:= -\big(FJ_A\big)^{\rm s}$ is represented in the frame $(\frac{\partial}{\partial \ang_i}, \frac{\partial}{\partial \mu_j})$ on $TM^0$  by the Gram matrix
\begin{equation}\label{e:gB}
\begin{split}
\pg_{A,B}  \sim  & \left( \begin{array}{cc} 
 \big({\bf \Psi}^{-1}\big)^{\rm s} & {\bf \Psi}^{-1}B   \\
-B\big({\bf \Psi}^{\rm T}\big)^{-1}  & {\bf \Psi}^{\rm s}  \end{array} \right)  
=\left( \begin{array}{cc} 
 {\bf \Psi}^{-1}& 0 \\
 0 & {\rm Id}  \end{array} \right) \left( \begin{array}{cc} 
 {\bf \Psi}^{\rm s} & B \\
- B & {\bf \Psi}^{\rm s}  \end{array} \right) \left( \begin{array}{cc} 
 \big({\bf \Psi}^{\rm T}\big)^{-1} &0 \\
0 & {\rm Id} \end{array} \right).
\end{split}
\end{equation}
The equality of matrices in the above formula is deduced by using  \eqref{linear-algebra}. It follows that  $F$ tames $J_A$ on $M^0$  if and only if 
the Hermitian matrix 
\begin{align}\label{positivity}
 {\rm Hess}(u) + \sqrt{-1}B>0
 \end{align}
  is positive definite at any point of  $\mathring{\Pol}$.  A similar condition imposed on the interior of each face of $\Pol$ by virtue of Proposition~\ref{p:boundary} guarantees that $F$ tames $J$ on $M$. Notice that this is independent of the transformation of type $A$ used to obtain $J$.

\begin{defn}\label{d:typeB} The generalized K\"ahler structure $(\pg_{A,B}, I_{A,B}, J_A)$ on $M$ described above will be  called a  {\it generalized K\"ahler deformation of type $B$} of  $(\pg_A, I_A, J_A)$.
\end{defn}

By the results in  \cite[Sect. 4]{W2},  we have the following generalization of the description of toric K\"ahler metrics.

\begin{prop}\label{p:toric-GK}\cite{W2} Any $\T$-invariant generalized K\"ahler structure of symplectic type on $M$,  compatible with  a  $\T$-invariant symplectic  form $F$ on $M$ with associated  labelled Delzant polytope $(\Pol,  {\La})$,  is $\T$-equivariantly isomorphic to a generalized K\"ahler structure $(\pg_{A,B}, I_{A,B}, J_A)$,   obtained from a toric K\"ahler metric $(g_u, J_u, \Fo)$ corresponding  to a symplectic potential $u\in {\mathcal S}(\Pol, {\La})$, by generalized K\"ahler transformations of type $A$ and $B$.
\end{prop}

\begin{convention}\label{convention}
In view of  the above result, and  to ease the notation,  we shall  omit  the indices $A,B$ and  implicitly assume (without loss of generality) that a toric generalized K\"ahler structure  $(g, I, J)$  compatible with a $\T$-invariant symplectic form $F$  on $M$ and corresponding  Delzant polytope  $(\Pol, {\La})$   is  {\it always} obtained from a $\T$-invariant K\"ahler structure  $(\kg_u, \kJ_u)$,   via generalized K\"ahler deformations of type $A$ and $B$ for respective elements $A,B \in \gL^2\tor$,  where  $u\in \mathcal{S}(\Pol, {\La})$, $\Fo$ is  a symplectic form  associated to  $(\Pol, {\La})$ via the Delzant construction,  and $(\kg_u, \kJ_u)$ is a $\Fo$-compatible K\"ahler metric with symplectic potential $u$ and canonical momentum-angular coordinates $(\mu,  \ang^c)$ on $(M^0, \Fo)$. 
\end{convention}

The computations  \eqref{tilde-y-inv}  and \eqref{pluriharmonic-I}  of the corresponding pluriharmonic coordinates of $J$ and $I$ underline the general symmetry of the construction:
\begin{equation}\label{symmetry-AB}
 (u, A, B) \longleftrightarrow (u, -A, -B)
 \end{equation}
 which, geometrically, corresponds to switching the roles of $J$ and $I$ of the corresponding toric GK structure $(g, I, J)$. Indeed,  this becomes apparent if we express the construction in  the momentum/angular coordinates $(\mu, \ang_F)$ of $F$. Then,  the corresponding K\"ahler structures $\kom_{\kJ_u}$ and $\kom_{\kI_u}$ are  linked to the common symplectic form $F$ via the relations
 \begin{equation}\label{symmetry}
 \begin{split}
 \kom_{\kJ_{u}} &= \sum_{j=1}^m \d\mu_j \wedge \d\ang_j = F -\sum_{i,j=1}^m B_{ij} \d \mu_i \wedge \d \mu_j\\
 \kom_{\kI_u} &= \sum_{j=1}^m \d \mu_j \wedge \d \bar \ang_j = F + \sum_{i,j=1}^m B_{ij} \d\mu_i \wedge \d \mu_j. \\
 F&= \tfrac{1}{2}(\kom_{\kI_{u}} + \kom_{\kJ_{u}}),
 \end{split}
 \end{equation}
where we recall $\bar \theta$ are defined in \eqref{pluriharmonic-I}.

\subsubsection{The case \texorpdfstring{$A=0$}{A=0}}\label{s:A=0}

When $A\neq 0$, the integrable almost complex structures $J=J_{A}$ and $\kJ_u$ associated respectively to the generalized K\"ahler structure $(\pg, I, J)$ and the K\"ahler structure $(\kg_u, \kJ_u)$ are different (even though  they are $\T$-equivariantly biholomorphic by virtue of Lemma~\ref{l:complex-identification}). In the special case $A=0$, we have $\kJ_u= J$ and $\kI_{u}=I$. We shall later make use of this simplification, so we recast below some facts specific to this case.  

First, letting
\begin{equation}\label{pi-B}
\Pi:=-\sum_{i,j=1}^m B_{ij} \frac{\partial}{\partial \ang_i}\wedge \frac{\partial}{\partial \ang_j} \end{equation}
denote the real Poisson structure on $M$ determined by $B\in \gL^2\tor$, the formula  \eqref{F-B} yields the following basic relation between  $\kom$ and $\pom$ on $(M, J)$:
\begin{equation}\label{non-linear-link}
F = (\kom^{-1} + \Pi)^{-1}, \qquad \pom = F^{1,1}.
\end{equation}
It will be also convenient in this case to express the geometry in terms of the $J$-pluriharmonic coordinates $(y_j, \ang_j)$ obtained via the Legendre transform of $u$ (see \eqref{Legendre} and Proposition~\ref{p:symplectic-to-complex}).  Setting  $y_j = u_{, j}(\mu)$ and $\phi(y)+ u(\mu)=\sum_{j=1}^m \mu_j y_j$, we have on $(M^{0}, J)$ (see \eqref{Kahler-toric-potential}):
\[\kom =  \d \d^c \phi = \sum_{i,j=1}^m \Big({\rm Hess}(\phi(y)) \Big)_{ij} \d y_i \wedge d\ang_j,\]
whereas the symplectic $2$-form $F$ and the fundamental $2$-form $\pom= F^{1,1}$ are given by:
\begin{equation}\label{omega-beta}
\begin{split}
F &=  \sum_{i,j=1}^m \Big(\phi_{, ij} \d y_i \wedge \d\ang_j  + \Big[\big({\rm Hess}(\phi)\big)B \big({\rm Hess}(\phi)\big)\Big]_{ij} \d y_i \wedge \d y_j \Big) \\
              &= \kom + \kom \Pi \kom; \\
\pom &= \sum_{i,j=1}^m \Big(\phi_{, ij}  \d y_i \wedge \d\ang_j  + \frac{1}{2}\Big[\big({\rm Hess}(\phi)\big) B \big({\rm Hess}(\phi)\big)\Big]_{ij} (\d y_i \wedge \d y_j + \d\ang_i \wedge \d \ang_j)\Big) \\
            &= \kom + \kom \left(\Pi^{1,1}\right) \kom.
\end{split}
\end{equation}
In this case,  $(y_j, \bar\ang_j)$ are pluriharmonic coordinates of $I$ (see \eqref{pluriharmonic-I}), so similar formulae hold with respect to $I$ (in which one should replace $B$ by $-B$ and $\ang_j$ by ${\bar \ang}_j$). We also notice that 
\begin{equation}\label{IJ-potential}
\d \d^c_J \phi= \kom_J, \qquad \d \d^c_I \phi = \kom_I,\end{equation}
so, by \eqref{symmetry}, we get
\begin{equation}\label{F-potential}
F= \tfrac{1}{2}(\d\d^c_J \phi + \d\d^c_I \phi).
\end{equation}

\subsubsection{The holomorphic Poisson structure of a toric generalized K\"ahler structure}

The next result gives a geometric meaning of the parameters $A, B\in \gL^2 \tor$.  In particular, they completely determine the associated holomorphic Poisson structures.

\begin{prop}\label{p:toric-Poisson} Let $(\pg, I, J)$ be a toric GK structure associated to the data $(u, A, B$).  Then the corresponding holomorphic Poisson structure on $(M, J)$ given by \eqref{GK-Poisson} is 
\[
\sigma_J=  2\sum_{i,j=1}^m \big(A_{ij} + \sqrt{-1}B_{ij} \big)\big(K_i -\sqrt{-1} JK_i\big) \wedge \big(K_j - \sqrt{-1}JK_j\big),
\]
where $K_1, \ldots, K_m$ are the fundamental vector fields associated to a basis of $\tor$,  and $A_{ij}$ and $B_{ij}$ are the corresponding components of $A$ and $B$ in this basis.
\end{prop}
\begin{proof}  The above formula follows by straightforward  but somewhat tedious matrix computations,  using  that $K_j = \frac{\partial}{\partial \ang_j}$,  \eqref{e:IB},  the identification (which  in turn follows from \eqref{e:gB})
\begin{equation*}
\pg^{-1}  \sim  \left( \begin{array}{cc} 
 {\bf \Psi}^{\rm T} & 0 \\
 0 & {\rm Id}  \end{array} \right) \left( \begin{array}{cc} 
 \X & \Y \\
 -\Y & \X \end{array} \right) \left( \begin{array}{cc} 
 {\bf \Psi} &0 \\
0 & {\rm Id} \end{array} \right), \end{equation*}
where
\begin{equation*}
\X:= \Big({\bf \Psi}^{\rm s} + B\big({\bf \Psi}^{\rm s}\big)^{-1}B\Big)^{-1}, \,\  \Y :=- \big({\bf \Psi}^{\rm s}\big)^{-1}B\X,
\end{equation*}
and the identities \eqref{linear-algebra}. \end{proof}

\subsection{The Bismut-Ricci form of a toric generalized K\"ahler structure}
M.~Abreu~\cite{Abreu1} obtained the expression of the Ricci form  $\rho_J$ of a  K\"ahler metric  $(g, J)\in \mathcal{K}^{\T}(M, \Fo)$ in terms of the momentum/angular coordinates and the corresponding $u\in \mathcal{S}(\Pol, {\La})$ as follows
\begin{equation}\label{Kahler-Ricci-form}
\rho_{J}  = \tfrac{1}{2}\d\d^c_{J} \Big(\log \det \big({\rm Hess}(u)\big) (\mu)\Big)=-\tfrac{1}{2} \sum_{i,j,k=1}^m  ({\bf G}^{-1})_{ij,ik}(\mu) \d\mu_k \wedge \d \ang_j.
\end{equation}
In this subsection we derive the Bismut-Ricci forms $\rho^B_I, \rho^B_J$ associated to a toric generalized K\"ahler structure $(\pg, I, J)$, using momentum-angular coordinates of $\som$, thus extending \eqref{Kahler-Ricci-form}.
\begin{lemma}\label{Ricci} Let $(\pg, I, J)$ be a toric generalized K\"ahler structure, associated to the data $(u, A, B)$. Then,  the Bismut-Ricci forms $\rho_I^B, \rho^B_J$  of $(\pg, J)$  are given on $M^0$ by
\begin{equation*}
\begin{split}
\rho^B_I  = & \ \tfrac{1}{2}\d\d^c_J \Big(\log  \det \big({\rm Hess}(u)\big)(\mu) + \sqrt{-1}B \Big),\\
\rho^B_J  =&\ \tfrac{1}{2}\d\d^c_I \Big(\log  \det \big({\rm Hess}(u)\big)(\mu) + \sqrt{-1}B \Big).
\end{split}
\end{equation*}
\end{lemma}
\begin{proof}  We are going to use Proposition~\ref{p:symplectic-Ricci-potential} with respect to the holomorphic section (defined on $M^0$)  
\[\Theta_J := \left((K_1- \sqrt{-1}JK_1) \wedge \cdots \wedge (K_m - \sqrt{-1}JK_m)\right)^{-1}.\]
To this end, recall that $K_j =\frac{\partial}{\partial \ang_j}$ and,  by \eqref{e:IB} and  \eqref{e:gB},   we have (in momentum/angular coordinates)
\begin{equation}\label{omegaJ}
\pom_J \sim  \left( \begin{array}{cc} 
{\bf \Psi}^{-1}B\big({\bf \Psi}^{\rm T}\big)^{-1} & - {\bf \Psi}^{-1}{\bf \Psi}^{\rm s} \\
 {\bf \Psi}^{\rm s} \big({\bf \Psi}^{\rm T}\big)^{-1} & B \end{array} \right), \end{equation}
so  we compute,   by using the above together with  \eqref{e:gB}, \eqref{omegaJ}, \eqref{linear-algebra}
\begin{align*}
\begin{split}
2^{-m}\left(\frac{\pom_J^{[m]}}{v_{\Theta}}\right) & = \det \Big(g(K_i, K_j) - \sqrt{-1}\omega_J(K_i, K_j)\Big) \\
                                      &=\det \Big(\big({\bf \Psi}^{\rm T}\big)^{-1}\big( {\bf \Psi}^{\rm s} - \sqrt{-1}B\big) {\bf \Psi}^{-1}\Big) =  \frac{\det\big({\rm Hess}(u) + \sqrt{-1}B\big)}{\big(\det {\bf \Psi}\big)^2},
                                      \end{split}
\end{align*}
As $\Fo^{[m]} =F^{[m]}$ (see \eqref{omega} and \eqref{F-B}), using \eqref{omega} and \eqref{omegaJ}, we also get
 \[ \left(\frac{F^{[m]}}{\pom_J^{[m]}}\right) =\left(\frac{\Fo^{[m]}}{\pom_J^{[m]}}\right)= \frac{\det{\bf \Psi}}{\det\left(\Hess(u) + \sqrt{-1}B\right)},\]
 and therefore
 \begin{equation}\label{potential-computation}
 \begin{split}
 \left[\log\left(\frac{F^{[m]}}{v_{\Theta}}\right) + \log\left(\frac{F^{[m]}}{\omega_J^{[m]}}\right)\right] &=\left[2\log\left(\frac{F^{[m]}}{\omega_J^{[m]}}\right) + \log\left(\frac{\pom_J^{[m]}}{v_{\Theta}}\right)\right] \\
 &=-\log \det \left( \Hess(u) + \sqrt{-1}B\right) + \mbox{const}.
 \end{split}
 \end{equation}
This yields the formula for $\rho^B_I$ via Proposition~\ref{p:symplectic-Ricci-potential}.  Using the symmetry \eqref{symmetry-AB} and that $\det({\rm Hess}(u) + \sqrt{-1}B)= \det({\rm Hess}(u) - \sqrt{-1}B)$, we get the formula for $\rho^B_J$. \end{proof}

We now restrict to the special case when the compact toric symplectic manifold $(M, \Fo, \T)$ satisfies the (topological) Fano condition $2\pi c_1(M, J_u) = [\Fo]$ for any $J_u \in \mathcal{K}^{\T}(M, \Fo)$.   It is well-known that  this  is equivalent to the assumption that $(\Pol, {\La})$ satisfies 
\begin{align}\label{Fano}
 L_j(0) = 1, \quad \, j=1, \ldots d.
 \end{align}
 We refer to a Delzant polytope  $(\Pol, {\La})$ satisfying \eqref{Fano} as {\it barycentred}.   In this setting,  by \eqref{Kahler-Ricci-form}  and \eqref{Kahler-toric-potential}, we have for any $\Fo$-compatible toric K\"ahler structure $(g_u, J_u)$
\begin{equation}\label{ricci-potential}
\rho_{J_u} - \Fo  = \d\d^c \Big(\tfrac{1}{2} \log \det \big({\rm Hess}(u(\mu))\big)  - (-u(\mu) + \sum_{i=1}^m \mu_i u_{,i}(\mu))\Big).
\end{equation}
Using the boundary conditions for $u$ in Proposition~\ref{p:boundary} (which imply that $\det{\rm Hess}(u) = \frac{\delta(x)}{\prod_{j=1}^d L_j(x)}$ for a positive smooth function $\delta$ on $\Pol$, see \cite{Abreu2})  and the Fano condition \eqref{Fano}, the function 
\begin{equation}\label{h-u}
h_u(\mu) :=\tfrac{1}{2}  \log \det \big({\rm Hess}(u(\mu))\big) - (-u(\mu) + \sum_{i=1}^m \mu_i u_{,i}(\mu))
  \end{equation}
appearing  at the RHS of  \eqref{ricci-potential}  is smooth on $M$. It thus follows that $\rho_{J_u}-  \Fo =  \d\d^c_{J_u} h_u$ \emph{globally} on $M$.

Similar to the K\"ahler toric Fano case, a generalized K\"ahler toric structure corresponding to data $(u,A,B)$ also admits a relative Ricci potential:

\begin{lemma}\label{l:global-potential} Suppose $(\Pol, {\La})$ is a barycentred Delzant polytope corresponding to a toric Fano variety. Then, for any toric GK structure $(\pg, I,J)$ corresponding to the data $(u,A,B)$, the function 
\[
\begin{split}
{\tih}_{u, B}(x) &:= \tfrac{1}{2}\log\det\left({\rm Hess}(u(x)) + \sqrt{-1}B\right) - \big(-u(x) + \sum_{j=1}^mx_j  u_{,j}(x)\big)\\
                       &= \tfrac{1}{2}\log\det\left(\big({\rm Hess}(\phi(y))\big)^{-1} + \sqrt{-1}B\right)  -\phi(y) \end{split} \]
pulls-back by the momentum map $\mu$ of $F$  to a smooth function on $M$. Furthermore, if $A=0$ we have
\[
\begin{split}
&\rho^B_J - \kom_{I} = \d \d^c_{I}(h_{u,B}(\mu)),\qquad  \rho^B_I - \kom_J = \d \d^c_J(h_{u, B}(\mu)),\\
&\tfrac{1}{2} (\rho^B_J + \rho^B_I)- F = \tfrac{1}{2}\big(\d \d^c_I (h_{u,B}(\mu)) + \d\d^c_J (h_{u,B}(\mu))\big).
\end{split} \]
\end{lemma}
\begin{proof} The smoothness of  ${\tih}_{u, B}(\mu)$ follows from the smoothness of $h_u(\mu)$ in  \eqref{h-u},  and the fact that $\det({\rm Hess}(u))/\det({\rm Hess}(u) + \sqrt{-1}B)$ is smooth and positive on $M$ as it computes the ratio $F^m/\omega_J^m$, where $(F, \omega_J)$ is the GK structure obtained by a $B$-transformation of the K\"ahler structure $({\overline g}_u, \Fo)$.  The Bismut-Ricci curvature formulas follow from Lemma~\ref{Ricci}, \eqref{IJ-potential}, and \eqref{F-potential}. \end{proof}

\section{The normalized generalized K\"ahler-Ricci flow on a toric Fano variety} \label{s:LTE}
\subsection{Scalar Reduction} 

We now describe the NGKRF \eqref{e:NGKRF} with $\gl = 1$, starting from a toric generalized K\"ahler structure $(g, I, J)$ on a smooth toric Fano manifold $(M, \som)$. We suppose that the initial symplectic form $F$ belongs to $2\pi c_1(M, J_0)$,  and that  $(g, I, J)$ satisfies Convention~\ref{convention} (so  we have $\kom=\som  \in 2\pi c_1(M, J)$ as well).  The solution of  \eqref{e:NGKRF} is defined in terms of  Hermitian metrics $(g_t, J)$ on a fixed complex manifold $(M, J)$, whereas Convention~\ref{convention} normalizes (via the action of a $\T$-equivariant diffeomorphism)  the toric GK structures on $M$  to be obtained from a K\"ahler structure $(\kg_u, \kJ_u)$ compatible with the fixed symplectic form $\som$ and canonical angular coordinates via deformations of type $A$ and $B$. This is similar to the symplectic versus complex viewpoints in the K\"ahler case (see Proposition~\ref{p:symplectic-to-complex}),  and  the  $\T$-equivariant diffeomorphism  in Lemma~\ref{l:complex-identification}  connects the  two points of view in the generalized K\"ahler case. The next proposition makes this explicit.

\begin{prop}\label{p:general-reduction} Let $(g_0, I_0, J_0)$ be a toric GK structure on a Fano toric variety, corresponding to the data $(u_0, A_0 , B_0)$  with respect to a symplectic form $\Fo \in 2\pi c_1(M, J_0)$. 
Denote by $(g_t, I_t, J_t)$ the family of GK structures corresponding to the solution $(u_t, A_t, B_t)$ of
\begin{equation}\label{toric-GKRF}
\begin{split}
\frac{\partial}{\partial t} u =&\ \log \Big( {\det} \big({\rm Hess}(u) + \sqrt{-1}e^{-2t} B_0\big)\Big) - 2\Big(-u +\sum_{j=1}^m x_j u_{,j} \Big),\\
A_t =&\ e^{-2t}A_0 ,  \qquad  B_t= e^{-2t}B_0.
\end{split}
\end{equation}
Then $\Phi_t^*(g_t,I_t,J_t)$ solves~\eqref{e:NGKRF}, where $\Phi_t^*$ is the canonical biholomorphism between $(M,J_0)$ and $(M,J_t)$ from Lemma~\ref{l:complex-identification}.
\end{prop}
\begin{proof} We shall do computations on $M^0$, using by Lemma~\ref{l:global-potential} that for $u_t\in \mathcal{S}(\Pol, {\La})$ the RHS of \eqref{toric-GKRF} is a smooth function on $\Pol$, so  \eqref{toric-GKRF} is a flow of elements $u_t\in \mathcal{S}(\Pol, {\La})$ defined  for $t\in [0, T)$,  thus giving  rise to  a smooth path of GK structures $(g_t, I_t, J_t), t\in [0, T)$  on $M$, obtained as deformations of type $A$ and $B$ of the K\"ahler structures $(g_{u_t}, J_{u_t})$ in $\mathcal{K}^{\T}(M, \Fo)$ with the $2$-vectors $A_t, B_t \in \gL^2 \tor$.  

We denote by $J_t$ the  complex structure defined by the deformation of type $A$ of the complex structures  $J_{u_t}\in \mathcal{K}^{\T}(M, \Fo)$  with  the $2$-vector  $A_t= e^{-2t}A_0\in \gL^2\tor$,  and by  $F_t= \som+ e^{-2t}\langle \d \mu,  B_0, d\mu\rangle$ the symplectic form of $(g_t, J_t, I_t)$, with $I_t$ being  the $F_t$-conjugate  of the almost complex structure $J_t$. Denote by $\tilde y_t$  the $J_t$-pluriharmonic variables \eqref{tilde-y}, and  consider the $\T$-equivariant  diffeomorphism  $\Phi_t$   defined in Lemma~\ref{l:complex-identification}, which satisfies  $\Phi_t (\tilde y_t) = \tilde y_0$  and $\Phi_t^*(J_t) = J_0=:J$.  We let $\tilde F_t : = \Phi_t^*(F_t)$, so that $\tilde F_t$ is a symplectic form on $(M, J)$ which gives rise to 
a GK structure $(\tilde\pg_t, \tilde I_t, J)$ compatible with $\tilde F_t$.  We will check that with $(u_t, A_t, B_t)$ satisfying \eqref{toric-GKRF}, the symplectic forms $\tilde F_t$ satisfy 
\begin{equation}\label{symplectic-generalized K\"ahler Ricci flow}
\frac{\part}{\part t} \tilde F_t = - 2\rho_J^B (\tilde \pg_t)  + 2 \tilde F_t.
\end{equation}
Taking the $(1,1)$ part with respect to $J$ of \eqref{symplectic-generalized K\"ahler Ricci flow}  yields \eqref{e:NGKRF} with $\lambda =1$.

For a toric GK structure $(\pg, I, J)$,  corresponding to the data  $(u, A, B)$, we consider the pluriharmonic variables $\tilde y$ with respect to $J$, defined in \eqref{tilde-y}, and a function $\tilde\phi(\tilde y)=\phi(y)$,  introduced on $M^{0}$ by 
\begin{equation}\label{tilde-phi}
\tilde \phi (\tilde y) := -u(\mu)  +\langle\mu, \tilde y\rangle =  -u(\mu) +\langle\mu, y\rangle= \phi(y), 
\end{equation}
where $y_j = u_{,j}(\mu)$,    and $\phi(y)$ is  the Legendre transform  \eqref{Legendre} of $u$. Using \eqref{e:IB}
we obtain the following formula (the detailed verification of which we leave to the reader)

	\begin{equation}\label{J-Legendre}
	\d\d^c_J  \tilde \phi  =  F + \d J \Big(\sum_{i,j=1}^m  A_{ij} \mu_i  \d\mu_j\Big) -  \d\Big(\sum_{i,j=1}^m B_{ij} \mu_i \wedge \d\mu_j\Big).\\
	\end{equation}
Using \eqref{J-Legendre},  we have that
\begin{equation*}\label{e:F_t}
\begin{split}
\tilde F_t =&\ \Phi_t^*\Big(\d\d^c_{J_t} {\tilde \phi}_t({\tilde y}_t) -\d J_t\big(\sum_{i,j=1}^m \big(A_t)_{ij}\mu_i \d\mu_j\big) + \sum_{i,j=1}^m(B_t)_{ij}\d\mu_i \wedge \d\mu_j\Big)\\
                =&\ \d \d^c_J {\tilde \phi}_t({\tilde y}_0) -\d J \big(\sum_{i,j=1}^m \big(A_t)_{ij}\mu^t_i \d\mu^t_j\big) + \sum_{i,j=1}^m(B_t)_{ij}\d\mu^t_i \wedge \d\mu^t_j\Big),
                \end{split}
                \end{equation*}
where $\mu^t := \mu \circ \Phi_t$.
Using \eqref{toric-GKRF}, we  get
\begin{equation}\label{evolution}
\begin{split}
\frac{\partial}{\partial t} \tilde F_t  = &\ \frac{\partial}{\partial t}\Big(\d \d^c_{J} \tilde\phi_t({\tilde y}_0)   - \d J\big(\sum_{i,j=1}^{m} (A_t)_{ij}\mu_i^t \d\mu_j^t\big) + \sum_{i,j=1}^m(B_t)_{ij}\d\mu^t_i \wedge \d\mu^t_j\Big)  \\
                                           =&\ \d J \d\Big( \frac{\partial \tilde\phi_t}{\partial t}(\tilde y_0) \Big)  + 2  \d J\Big(\sum_{i,j=1}^{m} (A_t)_{ij} \mu_i^t \d\mu_j^t\Big) - 2\sum_{i,j=1}^m(B_t)_{ij}\d\mu^t_i \wedge \d\mu^t_j  \\
                                             &\ \qquad- \d J\Big[\sum_{i,j=1}^{m} \Big((A_t)_{ij} \big(\frac{\partial {\mu}_i^t}{\partial t}\big) \d\mu_j^t +  (A_t)_{ij}{\mu}_i^t \d\big(\frac{\partial {\mu}_j^t}{\partial t}\big)\Big)\Big] \\
     &\ \qquad + \d \Big[\sum_{i,j=1}^{m} \Big((B_t)_{ij} \big(\frac{\partial {\mu}_i^t}{\partial t}\big) \d\mu_j^t +  (B_t)_{ij}{\mu}_i^t \d\big(\frac{\partial {\mu}_j^t}{\partial t}\big)\Big)\Big]                                        .
\end{split}
\end{equation}
We thus need to compute the $t$ derivatives  of  $\tilde \phi_t$ and $\tilde \mu^t$.  We start with computing the $t$ derivative of $\tilde \phi_t$.  By  \eqref{tilde-y},  \eqref{tilde-phi} and \eqref{toric-GKRF}, we first compute
\begin{equation}\label{y-dynamic}
\frac{\partial \tilde y_r^t}{\partial t}  = \big(\log\det \big({\rm Hess}(u_t) + \sqrt{-1}B_t\big)\big)_{, r} - 2\sum_{\ell=1}^m \big({\bf \Psi}_t^{\rm T}\big)_{r\ell} \mu_\ell.
\end{equation}
Therefore, we have
\begin{equation}
\begin{split}
\frac{\partial \tilde \phi_t}{\partial t}(\tilde y_t)  =&\ - \sum_{r=1}^m \Big(\frac{\partial {\tilde \phi}_t}{\partial {\tilde y}_r}\Big)\Big(\frac{\partial \tilde y_r}{\partial t}\Big)  -\frac{\partial u_t}{\partial t}(\mu)  + \sum_{r=1}^m \mu_r\Big(\frac{\partial \tilde y_r}{\partial t}\Big)\\
=&\ - \frac{\partial u_t}{\partial t}(\mu)   - \sum_{k,r=1}^m  \big(A_t\big({\bf \Psi}_t^{\rm T}\big)^{-1}\big)_{kr} \big(\log \det \big({\rm Hess}(u_t)+ \sqrt{-1}B_t\big)\big)_{, r} \mu_k \\
=&\ - \log \det \big({\rm Hess} (u_t) + \sqrt{-1}B_t\big) + 2\tilde \phi_t(\tilde y_t) \\
&\ \qquad - \sum_{k,r=1}^m  \big(A_t{\big(\bf \Psi}_t^{\rm T}\big)^{-1}\big)_{kr} \big(\log \det \big({\rm Hess}(u_t) + \sqrt{-1}B_t\big)\big)_{, r} \mu_k,
\end{split}
\end{equation}
where for passing from the first line to the second we have used \eqref{y-dynamic} and the fact that $A_{ij}$ is skew-symmetric, and for passing from the second line to the third we have used \eqref{toric-GKRF} and the fact that $\tilde \phi_t(\tilde y_t)= \phi_t(y_t)$ (see \eqref{tilde-phi}).

We now compute  the derivative of $\mu^t= \mu \circ \Phi_t$.  Using that $\mu_j^t(\tilde y_t)= \mu_j$ and \eqref{y-dynamic}, we get
\begin{equation}\label{mu}
\begin{split}
\frac{\partial {\mu}_j^t}{\partial t}(\tilde y_t)  
                                    &= 2\mu_j  - \sum_{r=1}^m \big({\bf \Psi}_t^{\rm T}\big)^{-1}_{jr}\big(\log{\det}\big({\rm Hess}(u_t) + \sqrt{-1}B_t\big)\big)_{, r}. 
\end{split}
\end{equation}
Thus, substituting back in \eqref{evolution} and using ~\eqref{J-Legendre}, we obtain  
\begin{equation*}\label{crucial}
\begin{split}
& \big(\Phi_t^*)^{-1}\Big( \frac{\partial}{\partial t} \tilde F_t\Big) \\
& \qquad = \d  \d^c_{J_t}\Big[2\tilde \phi_t -  \log \det \big({\rm Hess}(u_t) + \sqrt{-1}B_t\big) - \sum_{k,r=1}^m \big(A_t\big({\bf \Psi}_t^{\rm T}\big)^{-1}\big)_{kr} \mu_k \big(\log \big(\det{\rm Hess}(u_t) + \sqrt{-1}B_t\big)\big)_{, r}\Big]\\
                                                                                            &\qquad \quad -2 \d J_t\Big[\sum_{k,r=1}^m (A_t)_{kr}\mu_k \d\mu_r\Big]  + 2\d\Big[\sum_{k,r=1}^m (B_t)_{kr}\mu_k \d\mu_r\Big] \\
                                                                                           & \qquad \quad
                                                                                            - \d J_t\Big[\sum_{k,r=1}^m \big(A_t\big({\bf \Psi}_t^{\rm T}\big)^{-1}\big)_{kr}\big(\log \det \big({\rm Hess}(u_t) + \sqrt{-1}B_t\big)\big)_{, r} \d\mu_k\Big] \\
                                                                                            &\qquad \quad + \d J_t\Big[\sum_{k,r=1}^m   \mu_k \d\Big(\big(A_t\big({\bf \Psi}_t^{\rm T}\big)^{-1}\big)_{kr}\big(\log\det\big({\rm Hess}(u_t)+\sqrt{-1}B_t\big)\big)_{,r}\Big)\Big] \\
                                                                                            &\qquad \quad + 2\d \Big[\sum_{k,r=1}^m \big(B_t\big({\bf \Psi}_t^{\rm T}\big)^{-1}\big)_{kr}\big(\log \det \big({\rm Hess}(u_t) + \sqrt{-1}B_t\big)\big)_{, r} \d\mu_k\Big].
                                                                                            \end{split}
                                                                                            \end{equation*}
Furthermore, using Lemma~\ref{Ricci} and \eqref{e:IB}, we finally have
\begin{equation*}
\begin{split}
& \big(\Phi_t^*)^{-1}\Big( \frac{\partial}{\partial t} \tilde F_t\Big)  = 2F_t  -\d J_t \d \Big[\log \det\big( {\rm Hess}(u_t) + \sqrt{-1}B_t\big)\Big]  \\
                                                                                        &\qquad \quad - 2 \d\Big[\sum_{\ell,r=1}^m \Big({\bf \Psi}_t^{-1} A_t\big({\bf \Psi}_t^{\rm T}\big)^{-1}\big)_{\ell r} \big(\log\det\big({\rm Hess}(u_t) + \sqrt{-1}B_t\big)\big)_{,r} \d \ang_{\ell}\Big) \Big] \\
                                                                                        &\qquad \quad + 2\d \Big[\sum_{k,r=1}^m \big(B_t\big({\bf \Psi}_t^{\rm T}\big)^{-1}\big)_{kr}\big(\log \det \big({\rm Hess}(u_t) + \sqrt{-1}B_t\big)\big)_{, r} \d\mu_k\Big] \\
                                                                                        &\qquad = 2F_t  -\d J_t\d\Big[\log \det \big({\rm Hess}(u_t) + \sqrt{-1}B_t\big)\Big] - 2 \d\Big[\sum_{\ell,r=1}^m \big({\bf \Psi}_t^{-1}\big)^{\rm a}_{r\ell}\big(\log\det{\rm Hess}(u_t) + \sqrt{-1}B_t\big)_{,r} \d\ang_{\ell}\Big) \Big]\\
   &\qquad \quad - 2\d \Big[\sum_{k,r=1}^m \big({\bf \Psi}_t^{-1} B_t\big)_{rk}\big(\log \det \big({\rm Hess}(u_t) + \sqrt{-1}B_t\big)\big)_{, r} \d\mu_k\Big] \\
                                                                                  &\qquad = 2F_t  - \d {I_t} \d \Big[\log \det \big({\rm Hess}(u_t) + \sqrt{-1}B_t\big)\Big] = 2F_t -2 \rho^B_{J_t}(\pg_t),
\end{split}
\end{equation*}
as required. 
\end{proof}

We apply the Legendre transform \eqref{Legendre} to $u_t(x)$ and use the formula \eqref{Legendre-differential} to rewrite the evolution equation \eqref{toric-GKRF} of symplectic potentials  $u_t(x)$  as an evolution equation of  convex functions $\phi_t(y)$  on $\R^m$ (similar to the reduction \eqref{complex-KRF} of the normalized  K\"ahler-Ricci flow):

\begin{prop} \label{p:phireduction} Let $(u_t, A_t, B_t)$ denote a solution to (\ref{toric-GKRF}) on a toric Fano manifold.  The associated one-parameter family of potentials $\phi_t(y)$ on $\mathbb R^m$ satisfies
\begin{equation}\label{toric-GKRF-kahler}
\frac{\partial}{\partial t} \phi_t = \log  {\det}  \Big( \big({\rm Hess}(\phi_t)\big)^{-1} + \sqrt{-1}e^{-2t}B_0\Big)^{-1} + 2 \phi_t.
\end{equation}
\begin{proof} This follows from   \eqref{toric-GKRF} and the basic properties of the Legendre transform: ${\rm Hess}(\phi_t(y))= \Big({\rm Hess}(u_t(x))\Big)^{-1}$ and \eqref{Legendre-differential}.
\end{proof}
\end{prop}

The results of Proposition \ref{p:general-reduction} and \ref{p:phireduction} yield that the flow of $A$ completely decouples from the evolution equations for $u$ and $B$.  Because of this it suffices to study the flow in the case $A = 0$, an observation which greatly simplifies the technicalities to follow.

\begin{cor} \label{c:Adecoupling} Let $(g_0, I_0, J_0)$ be a toric GK structure on a Fano toric variety, corresponding to the data $(u_0, A_0 , B_0)$ with respect to a symplectic form $\som \in 2\pi c_1(M, J_0)$. Let $(u_t, 0, e^{-2t} B_0)$ be the unique maximal solution of (\ref{toric-GKRF}) with the initial data $(u_0, 0, B_0)$.
Denote by $(g_t,I_t,J_t)$ the generalized K\"ahler structure corresponding to the one-parameter family $(u_t, e^{-2t} A_0, e^{-2t} B_0)$, and by $\Phi_t^*$ the canonical biholomorphism between $(M,J_0)$ and $(M,J_t)$ from Lemma~\ref{l:complex-identification}. Then $\Phi_t^*(g_t,I_t,J_t)$ is the solution to \eqref{e:NGKRF} with initial data $(g_0, I_0, J_0)$.
\end{cor}

We first notice that when $B_0=0$, \eqref{toric-GKRF-kahler} reduces to the scalar reduction of the normalized K\"ahler Ricci flow on a toric Fano manifold (cf.~\cite{Zhu})
\begin{equation}\label{complex-KRF}
 \frac{\partial}{\partial t} \phi_t(y) = \log \det \big({\rm Hess}(\phi_t(y))\big)  + 2\phi_t(y).
 \end{equation}
In this case it follows from the fundamental works (cf.~\cite{TZ-07, TZ-13, WZ, Zhu}) that the solution to (\ref{complex-KRF}) exists for all time and converges in $C^{\infty}$ after translations to a smooth function $\psi_s$ defining a K\"ahler-Ricci soliton.

To study the general case when $B \neq0$, by now restricting to the case $A = 0$ using Corollary \ref{c:Adecoupling}, we are able to give an explicit description of the pluriclosed metric ${\pom}_t$ along the generalized K\"ahler Ricci flow in the toric setting,  as a certain transformation of a flow of associated K\"ahler metrics $\kom_t=\kom_{\phi_t}= \d\d_J^c \phi_t$ on a fixed complex manifold $(M, J)$.  The precise relation is given by \eqref{omega-beta}, i.e. 
\begin{equation}\label{toric-Kahler-reduction}
\pom_t = \kom_t + e^{-2t} \left[ \kom_t \left(\Pi_{B_0}^{1,1}\right) \kom_t\right], \end{equation}
where  $\kom_t:=\kom_{\phi_t}$ is the flow of K\"ahler metrics  corresponding to a solution $\phi_t$ of \eqref{toric-GKRF-kahler} and $\Pi_{B_0}^{1,1}$ denotes the $(1,1)$-part of the real Poisson tensor  \eqref{pi-B}.
We next recapture the K\"ahler flow $\kom_t$ in terms of $\kom_0$-relative K\"ahler potentials (where $\kom_0$ can be the initial K\"ahler metric, or any other reference $\T$-invariant K\"ahler metric in $2\pi c_1(M, J)$).

\begin{prop}\label{GKRF-Kahler-geometrc} Let $(g_0, I_0, J)$ denote a toric generalized K\"ahler structure on $(M, J)$, with corresponding data $(u_0, 0, B_0)$,  and let $\kom_0 \in 2\pi c_1(M, J)$ denote an initial toric K\"ahler metric with Ricci potential $h_0$.  Suppose $\varphi_t$ satisfies
\begin{equation}\label{unnormalized}
\begin{split}
& \frac{\partial}{\partial t} \varphi_t = \log\left(\frac{\kom_{\varphi_t}^{[m]}}{\kom_0^{[m]}}\right)+ \log\left(\frac{\kom_{\varphi_t}^{[m]}}{\pom_{\varphi_t}^{[m]}}\right) + 2\varphi_t - 2h_0 + c_t, \qquad \varphi(0)= 0,\\
& \kom_{\varphi_t}= \kom_0 + dd^c_J \varphi_t, \qquad \pom_{\varphi_t} = \kom_{\varphi_t} + e^{-2t} \left(\kom_{\varphi_t} \Pi_{B_0}^{1,1} \kom_{\varphi_t}\right). \end{split}
\end{equation}
Then, $\kom_t:= \kom_{\varphi_t}$ is a flow of K\"ahler metrics in $2\pi c_1(M, J)$,  and  $\pom_t= \pom_{\varphi_t}$ an associated flow of pluriclosed Hermitian metrics $\pg_t=-\pom_t J$ on $(M, J)$, defining GK structures $(\pg_t, I_t, J)$ with corresponding symplectic form
\[F_t :=  \kom_t + e^{-2t}\kom_t \Pi_{B_0} \kom_t, \]
and  satisfying   the evolution equations
\begin{equation}\label{geometric-evolution}
\frac{\partial}{\partial t} \kom_t = -2(\rho^B_{I_t}(\pg_t) -\kom_t), \qquad \frac{\partial}{\partial t} \pom_t = -2\left(\rho^B_J(\pg_t)^{1,1} -\pom_t\right). \end{equation}
\end{prop}
\begin{proof} We let $\varphi_t(y): = \phi_t(y) - \phi_0(y)$ which, by Proposition~\ref{p:symplectic-to-complex}, pulls back to  a smooth function on $M$ (still denoted by $\varphi_t$) such that  $ \kom_t = \kom_0 + \d\d^c_J \varphi_t$.  Using that  in the $(y_i, \ang_j)$ coordinates we have  (see \eqref{omega-beta})
\begin{equation*}
 \frac{\pom^{[m]}}{\kom^{[m]}} = \frac{\big(\det\big({\bf G} + \sqrt{-1}B\big)\big)\big(\det {\bf H}\big)^2}{\det {\bf H}_0}, \qquad  \frac{\kom^{[m]}}{\kom_0^{[m]}}   = \frac{\det {\bf H}}{\det {\bf H}_0},
 \end{equation*}
 where $\kom_0$ is any background toric K\"ahler metric (and ${\bf H}_0 = \kom_0(K_i, JK_j) = \Hess(\phi_0)$ is the corresponding matrix),  we can rewrite the flow \eqref{toric-GKRF-kahler}  in terms of $\varphi_t$ as \eqref{unnormalized}. The remaining claims follow from  Corollary~\ref{c:Adecoupling} and \eqref{toric-Kahler-reduction}. \end{proof}

\subsection{Evolution equations}

Here we determine the evolution of various quantities along NGKRF.  To begin we recall the Chern Laplacian and recall the explicit form for its action on $(p,0)$-forms.

\begin{defn}\label{d:ChernC} Let $(M, g, J)$ denote a Hermitian manifold, with associated Chern connection.  Let $\Delta^{C}_{\pg} : \Gamma(\Wedge^{p,q}(M)) \to \Gamma(\Wedge^{p,q}(M))$, denote the \emph{Chern Laplacian}, defined by
\[ \Delta^{C}_{\pg} \psi = \sum_{i=1}^m(\nabla^C)^2_{E_{i},\bar{E_{i}}} \psi, \]
where $E_i := e_i-\sqrt{-1}Je_i, \,  i=1, \ldots m$ for a $J$-adapted orthonormal frame $\{e_i, Je_i\}$.
\end{defn}

\begin{lemma} \label{l:ChernLap} Fix a Hermitian manifold $(M^{2n}, g, J)$ with Chern connection $\N^c$ and let $H= - d^c\omega$.  For $\psi \in \Wedge^{p,0}$, $\del \psi = 0$,  and $Z_1, \ldots Z_p \in T^{1,0}M$, one has
\begin{align*}
(\gD^C_{\pg} \psi)(Z_1, \ldots, Z_p) =\Big( \gD_g \psi  + H \star (\d\psi) + \d \left( H \star \psi  \right) - {\mathcal L}_{\theta^{\sharp}} \psi \Big)_{Z_1, \ldots, Z_p},
\end{align*}
where for a $(q+1)$-form $\psi$,  we let 
\[ (H \star \psi)_{X_1, \ldots X_q} : = \frac{1}{2}\sum_{i,k=1}^{2m} \sum_{j=1}^q (-1)^{j} H(e_i, e_k, X_j)\psi(e_i, e_k, X_1, \ldots, \hat X_j, \ldots, X_q). \]
\end{lemma}
\begin{proof} The proof is a straightforward computation left as an exercise.
\end{proof}

\begin{lemma}\label{linearization} Let $\kom_{\phi}  = \d\d^c_J \phi$ be a toric K\"ahler metric on $(M, J)$,  described by its toric K\"ahler  potential $\phi$ on $\R^m$, associated to a toric GK structure $(\pg_{\phi}, I_{\phi}, J)$ corresponding to $A = 0, B\in \Wedge^2\tor$.  Then, for a variation $\phi_s = \phi_0 + s \dot{\phi}$ of $\phi$ one has
\begin{align*}
\left. \frac{\del}{\del s} \right|_{s=0} \log \det \left( \left({\rm Hess}(\phi_s)\right)^{-1} + \sqrt{-1}B\right)^{-1} = \Delta^{C}_{g_{\phi}} \dot{\phi}.
\end{align*} 
\end{lemma}

\begin{proof} We denote by ${\bf H}:= {\rm Hess}(\phi(y)) = \big({\rm Hess}(u(x))\big)^{-1}$ and notice that 
\[\big({\bf H}^{-1} + \sqrt{-1} B\big)^{-1} = \X + \sqrt{-1} \Y\]
with $ \X =\big({\bf H}^{-1} + B{\bf H} B\big)^{-1}, \, \, \Y=-{\bf H} B \X$.
Thus, we have
\[
\begin{split}
&\delta_{\phi} \Big(\log \det \big( {\bf H}^{-1} + \sqrt{-1} B\big)\Big)(\dot{\phi}) = -{\rm tr}\Big(\big(\X + \sqrt{-1} \Y\big) {\bf H}^{-1} \dot{\bf H} {\bf H}^{-1} \Big)\\
                                                                                                               &\, = - {\rm tr} \Big({\bf H}^{-1} \X{\bf H}^{-1} {\rm Hess}(\dot{\phi})\Big) 
                                                                                                                =-{\rm tr} \Big(\big[{\bf H}[ {\bf H}^{-1} + B{\bf H}B  ] {\bf H}\big]^{-1}{\rm Hess}(\dot{\phi})\Big) \\
                                                                                                               &\, = -{\rm tr} \Big(\big[{\bf H} + {\bf H} B{\bf H}B {\bf H}\big]^{-1}{\rm Hess}(\dot{\phi})\Big).\                                                                                                    \end{split}
                                                                                                               \]
Notice that for any $\T$-invariant smooth function $f$ on $M$, viewed as a function $f(y)$ in the holomorphic coordinates $(y_j + \sqrt{-1} \ang_j)$, we have
\begin{equation*}\label{dd^c}
 \d\d^c_J f(y) = \sum_{i,j=1}^m \big({\rm Hess}(f(y))\big)_{ij} \d y_i \wedge \d\ang_j.
 \end{equation*}                                                                                                               
 whereas the inverse of $\pom_{\phi}$ is computed from  \eqref{omega-beta}  to be
 \[
 (\pom_{\phi})^{-1} = -\sum_{i,j=1}^m \tilde{\X}_{ij} \Big(\frac{\partial}{\partial y_i} \wedge \frac{\partial}{\partial \ang_j}\Big)  + \frac{1}{2}\tilde{\Y}_{ij} \Big(\frac{\partial}{\partial  y_i} \wedge \frac{\partial}{\partial y_j} + \frac{\partial}{\partial \ang_i }\wedge \frac{\partial}{\partial \ang_j}\Big)
 \]
 with $ \tilde{\X} =\big[{\bf H} + {\bf H} B{\bf H}B {\bf H}\big]^{-1}, \,  \tilde{\Y}= B{\bf H} \tilde{\X}$.
 We thus get 
 \[ \Delta^{C}_{g_{\phi}} f = - {\rm tr} \Big(\big[{\bf H} + {\bf H} B{\bf H}B {\bf H}\big]^{-1}{\rm Hess}(f(y))\Big)\]
 and the claim follows. \end{proof}
 
 \begin{lemma}\label{jeff-trick} Let $\varphi_t$ be a solution of \eqref{unnormalized}.  Then one has
 \begin{gather} \label{f:varphidotflow}
 \begin{split}
 \frac{\partial}{\partial t} \dot{\varphi}_t  &= \Delta^C_{\pg_t} \dot{\varphi}_t  + 2\dot{\varphi}_t -2|b_t|^2_{g_t}.
 \end{split}
 \end{gather}
\end{lemma}
\begin{proof} Notice that by Proposition~\ref{p:symplectic-to-complex}, $\varphi_t(y) =\phi_t(y)- \phi_0(y)$  extends to a smooth function on $M$, thus so does $\dot{\varphi}_t = \dot \phi_t$. We take derivative with respect to $t$  in \eqref{toric-GKRF-kahler} and use Lemma~\ref{linearization} to compute
\[
\begin{split}
\frac{\partial}{\partial t} \dot{\phi}_t =&\ \Delta^{C}_{\pg_t} \dot{\phi}_t   + 2{\rm tr}\big[ B_t{\bf H}_t B_t\X_t\big] + 2\dot{\phi}_t
                                                       =\ \Delta^{C}_{\pg_t} \dot{\phi}_t   +  2m -2{\rm tr}\big[{\bf H}_t^{-1}\X_t \big]+ 2\dot{\phi}_t\\
                                                       =&\ \Delta^{C}_{\pg_t} \dot{\phi}_t   +  2m -2{\rm tr}\big[{\rm Id} +  {\bf H}_t B_t {\bf H}_t B_t\big]^{-1} + 2\dot{\phi}_t, 
\end{split}
\]
where  we have set $B_t=e^{-2t}B_0$,  and ${\bf H}_t$ and $\X_t$ are the objects defined in the proof of Lemma~\ref{linearization} with respect to $\phi_t$.

As $F_tJ = -\pg_t -b_t$, we have 
\[|F_t|^2_{\pg_t} = \tfrac{1}{2}|F_tJ|^2_{\pg_t} = m + |b_t|^2_{\pg_t}, \]
noting that,  by our convention,  the norm of a $2$-form is half the norm of the corresponding tensor.
Using that $F_t = -2{\pg}_t (J+ I_t)^{-1}$ we obtain 
\begin{equation*}
2m + 2|b_t|^2_{\pg_t} =4|(J+ I_t)^{-1}|^2_{\pg_t} = -4{\rm tr}(I_t + J)^{-2} 
\end{equation*}
whereas  by  \eqref{e:IB} (recalling that $A=0$  and ${\bf H}= {\bf \Psi}^{-1}$ in our case) we compute
\[(J+ I_t)^2 
  \sim -4  \left( \begin{array}{cc} 
 {\rm Id} + B_t{\bf H}_t B_t  {\bf H}_t & 0 \\
 0 & {\rm Id} + B_t {\bf H}_t B_t {\bf H}_t
 \end{array} \right),\]               
so that  
\[m + |b_t|^2_{\pg_t} ={\rm tr}\big[{\rm Id} + B_t {\bf H}_t B_t {\bf H}_t\big]^{-1}.\]
The claim follows. \end{proof}
 
Using Lemma \ref{l:ChernLap}, we can refine the result of Lemma \ref{evolution-b}, giving a useful evolution equation for the $(2,0)$ piece of $b$, which yields a key differential inequality.
 
 \begin{prop}\label{p:b-evolution} Let $(M^{2n}, g_t, I_t, J)$ be a solution to GKRF with symplectic-type initial data, and set $\gb_J = \pi_{\Wedge^{2,0}_J} b$.  Then
\begin{align*}
\frac{\partial}{\partial t} \gb_J =&\  \gD^C_{\pg_t} \gb_J,\qquad
\frac{\partial }{\partial t} \brs{\gb_J}^2 =\ \gD^C_{\pg_t} \brs{\gb_J}^2 - \brs{\N \gb_J} - \brs{\delb \gb_J}^2 - \brs{\gb_J \hook H^{2,1}}^2.
\end{align*}
\begin{proof} We apply Lemma~\ref{l:ChernLap} to $\beta_J$.  In this case,   $H= \d b$  and  $\d \beta = H^{2,1}$,  so that  the term  $(H\star d\beta)_{Z_1, Z_2}$ is zero. By  Lemma~\ref{symplectic-type},  the term $b \righthalfcup H$ is exact, and thus  $(\d(H \star \beta))_{Z_1, Z_2} = (d(b \righthalfcup H))_{Z_1, Z_2} =0$. We conclude that
  \[\Delta^{C}_{\pg} \beta_J = \Big(\Delta_g \beta_J - \mathcal L_{\theta_J^{\sharp}} \beta_J \Big)^{2,0}.\]
As $\Delta_g - {\mathcal L}_{\theta_J^{\sharp}}$ is a real operator, it follows from Lemma~\ref{evolution-b} that the evolution equation of $\beta_t$ under the GKRF is  
 \[ 
  \frac{\partial} {\partial t} \beta_J = \Delta^C_{g_t} \beta_J, \]
which is the formula obtained in \cite{SBIPCF}, noticing that in this context $\beta_J = \del \alpha$ and $b = \Re(\partial \alpha)$ in the notation of \cite{SBIPCF}.  The evolution of $\brs{\gb_J}^2$ then follows from the established evolution for $\gb_J$ and a Bochner formula (cf. \cite{SPCFSTB} Lemma 4.7).
\end{proof}
\end{prop}

\subsection{Long-time existence}

In this subsection we give the proof of Theorem \ref{t:mainthm}.  We first derive an estimate for the tensor $b$ which holds without any symmetry hypotheses.  Next we derive two key a priori estimates for the potential $\varphi$ in the toric setting.  Using these and further maximum principle arguments we prove the long time existence.

\begin{prop}\label{p:basic-estimate} Let $(M^{2n}, g_t, I_t, J)$ be a solution to GKRF with symplectic-type initial data defined on $[0,T)$.  Then
\begin{align*}
\sup_{M \times [0,T)} \brs{b}^2 \leq \sup_{M \times \{0\}} \brs{b}^2.
\end{align*}
\end{prop} 
\begin{proof}
This follows by applying the maximum principle to the second relation in Proposition~\ref{p:b-evolution}, noting that $b =\Re(\beta_J)$.
\end{proof}

\begin{prop}\label{p:basic-estimate2} Let $\varphi_t$ denote a solution to (\ref{unnormalized}) on a toric Fano manifold.  There exists a constant $C > 0$ so that
\begin{align*}
\sup_{M \times \{t\}} \left( \brs{\varphi} + \brs{\dot{\varphi}} \right) \leq C e^{2t}.
\end{align*}
\end{prop}
\begin{proof} 
Applying the maximum principle directly to (\ref{f:varphidotflow}) yields the upper bound $\dot{\varphi} \leq C e^{2t}$.  By Proposition~\ref{p:basic-estimate}), we have  $|b_t|^2_{\pg_t} \leq C$.  Using this and applying the maximum principle to (\ref{f:varphidotflow}) yield the lower bound $\dot{\varphi} \geq - C e^{2t}$.  Using the estimate of $\dot{\varphi}$ the bound for $\varphi$ follows by integration in time.
\end{proof}

\begin{proof}[Proof of Theorem \ref{t:mainthm}] By Proposition \ref{p:toric-GK} the given invariant data is equivalent to a triple $(u_0, A_0, B_0)$.  By Corollary \ref{c:Adecoupling} the solution to NGKRF is described by a triple $(u_t, 0, e^{-2t} B_0)$, where $u_t$ is a solution to (\ref{toric-GKRF}).  Taking Legendre transform, by Proposition \ref{GKRF-Kahler-geometrc} this solution is equivalently described by $\varphi_t$ which solves (\ref{unnormalized}).  We will prove the long time existence of the flow $\varphi_t$ which in turn proves the long time existence of (\ref{e:NGKRF}) by the discussion above.

Equation (\ref{unnormalized}) is strictly parabolic for $\varphi$, thus the short-time existence is guaranteed by general theory.  To establish global existence, we first aim to establish uniform parabolicity of the equation, i.e. uniform equivalence of the time-varying metrics.  To obtain the upper bound 
we first recall that by \cite[Lemma~6.9]{AS}, we have
\begin{gather} \label{f:LTE10}
\bigg(\frac{\partial}{\partial t} - \Delta^C_{{\pg}_t}\bigg) \log \trace_{\kom_0} \pom_t \leq |T^c_t|^2_{\pg_t} + C\trace_{\pom_t}(\kom_0),\end{gather}
where $T^c$ is the torsion of the Chern connection.  To control the two terms on the right hand side of this inequality, we will augment our test function with two auxiliary terms.  We first observe that by Proposition \ref{p:b-evolution}, using that $db = H$ and dropping some negative terms we have
\begin{gather} \label{f:LTE20}
\left( \frac{\del}{\del t} - \gD^C_{\pg_t} \right) \brs{b_t}^2_{\pg_t} \leq - \brs{T^c}^2_{\pg_t}.
\end{gather}
Furthermore, observe that
\begin{gather} \label{f:LTE30}
\begin{split} 
\bigg(\frac{\partial}{\partial t} - \Delta^C_{\pg_t}\bigg)\varphi =&\ \dot{\varphi} - \trace_{\pom_t} \d \d^c_J \varphi
=\ \dot{\varphi} - \trace_{\pom_t} \kom_t + \trace_{\pom_t} \kom_{0}.
\end{split}
\end{gather}
Notice that
\begin{equation*}
\trace_{\pom_t}(\kom_t) =\trace\big[{\rm Id} + {\bf H}_t B_t {\bf H}_t B_t\big]^{-1} = m + |b_t|_{\pg_t}^2 \leq C,
\end{equation*}
where the last inequality follows from Proposition \ref{p:basic-estimate}.  Using this together with Proposition \ref{p:basic-estimate2} in (\ref{f:LTE30}) yields
\begin{gather} \label{f:LTE35}
-C + \tr_{\pom_t} \kom_0 \leq \left( \frac{\del}{\del t} - \gD_{\pg_t}^C \right) \varphi \leq C + \tr_{\pom_t} \kom_0.
\end{gather}
Now set
 \[ \Psi_1:= \log \trace_{\kom_0} \pom_t + A\Big(|b_t|^2_{\pg_t} + \varphi \Big). \] 
For $A$ chosen sufficiently large, it follows from (\ref{f:LTE10}), (\ref{f:LTE20}), and (\ref{f:LTE35}) that
 \[\bigg(\frac{\partial}{\partial t} - \Delta^C_{\pg_t}\bigg) \Psi_1 \leq CA.\]
 Applying the maximum principle, and using that $\brs{\varphi}$ is already bounded by Proposition \ref{p:basic-estimate2} we conclude that for any $T > 0$ there is a constant $C$ so that
 \[ \sup_{M \times [0,T]} \trace_{\kom_0} \pom_t \leq C. \]
To obtain the lower bound we get a lower bound for the volume form.  First, by \cite[Lemma~6.7]{AS}, we have
\begin{gather} \label{f:LTE40}
\left( \frac{\del}{\del t} - \gD^C_{\pg_t} \right) \log \frac{\pom_t^m}{\kom_0^m} \geq - C \tr_{\pom_t} \kom_0.
\end{gather}
Now set
\begin{gather*}
\Psi_2 = \log \frac{\pom_t^m}{\kom_0^m} + A \varphi.
\end{gather*}
By (\ref{f:LTE35}) and (\ref{f:LTE40}) we obtain
\begin{gather*}
\left( \frac{\del}{\del t} - \gD^C_{\pg_t} \right) \Psi_2 \geq - C.
\end{gather*}
Applying the maximum principle establishes the lower bound for the determinant.  It follows that $C^{-1}  g_0 \leq \pg_t \leq C g_0$ on any finite time interval.  As $|b_t|^2_{\pg_t}$ is bounded, we can now apply \cite[Theorem~1.2]{JS} to obtain the higher order regularity and hence the long time existence.
\end{proof}

\section{Monotone functionals and behavior at infinity} \label{s:MF}

In this section we analyze the behavior of NGKRF at infinity.  First we derive an extension of Perelman's shrinker entropy to the GKRF (without toric symmetry), and use it to derive the uniform $\gk$-noncollapsing estimate and convergence of nonsingular solutions of Theorem \ref{t:genconv}.  Next we derive an extension of the Mabuchi $K$-energy to the GK setting, and show that it is monotone along the GKRF with toric symmetry.  This monotonicity is the key point behind the convergence statements of Theorem \ref{t:convthm}.

\subsection{Entropy monotonicity}

In this subsection we observe that a modification of Perelman's entropy functional \cite{Perelman} is monotone along generalized K\"ahler-Ricci flow for structures of symplectic type.  Our discussion is a direct adaptation of \cite[Chapter 6]{GRFBook}, (cf. also \cite{ST3}).  There the monotonicity formulae and applications are contingent upon a further monotone quantity dubbed a torsion-bounding subsolution.  In the setting of generalized K\"ahler structures of symplectic type the torsion potential $b$ plays this role, and we make this explicit below

\begin{defn} \label{d:shrinkingentropy} Let $(M^{2m}, g, I, J)$ be a generalized K\"ahler structure of symplectic type.  
Fix $f\in C^\infty(M,\R)$ and $\tau>0$.  The \emph{shrinker entropy} associated to this data is
 \begin{align*} 
\WW(g,H,f,\tau)  := \int_M \left[ \tau \left( \brs{\N f}^2 + R - 
\frac{1}{12}
\brs{H}^2 \right) - \brs{b}^2 + f - 2m \right] (4\pi\tau)^{-m}e^{-f} dV.
\end{align*}
Furthermore, let
\begin{gather*}
\mu \left(g, H, \tau \right) := \inf_{\left\{ f\ | \int_M (4 \pi \tau)^{-m}e^{-f} dV = 1 
\right\} } \WW
\left(g, H, f, \tau \right),
\end{gather*}
\end{defn}

\begin{defn} \label{d:conjugateheat} Let $(M^{2m}, g_t, I_t, J)$ be a solution to GKRF with symplectic-type initial data.  The associated \emph{conjugate heat operator} is
\begin{align*}
\square^* = -\frac{\del}{\del t} - \gD_{g_t} - \tfrac{1}{2} \tr_g \dot{g}
\end{align*}
\end{defn}

\begin{prop} \label{modifiedmono-1} Let $(M^{2m}, g_t, I_t, J)$ be a solution to generalized K\"ahler-Ricci flow on a compact manifold with symplectic-type initial data.  Fix some $\tau^* > 0$ and set $\tau(t) = \tau^* - t$.  Let $u_t$ denote a solution of the conjugate heat equation $\square^* u = 0$, and define $f$ via $u = \frac{e^{-f}}{(4 \pi \tau)^{m}}$.
\begin{align*}
\frac{d}{d t} & \WW (g,H,f,\tau)
=& \int_M \left[ 2 \tau \brs{\Rc - \frac{1}{4} H^2
+ \N^2 f - \frac{1}{2 \tau}
g}^2 +\frac{\tau}{6} \brs{d^* H + i_{\N f} H}^2 + \frac{5}{6} \brs{H}^2
\right]
(4 \pi \tau)^{-m} e^{-f} dV.
\end{align*}
\begin{proof} This follows immediately from (\cite{GRFBook} Proposition 6.26), noting that $\brs{b}^2$ is a `torsion-bounding subsolution', which follows from Proposition \ref{p:b-evolution} and the fact that $\brs{\delb \gb_J}^2 = \brs{T^c}^2 = \brs{H}^2$ by Definition \ref{Bismutdef}.
\end{proof}
\end{prop}

\begin{rmk} Proposition \ref{modifiedmono-1} shows that the only self-similar solutions of the normalized GKRF with symplectic-type initial data are actually K\"ahler-Ricci solitons.  In contrast to the shrinking case, there exist non-K\"ahler steady solitons for GKRF \cite{Streetssolitons, SU1, SU2}.
\end{rmk}

\begin{proof}[Proof of Theorem \ref{t:genconv}] Using the entropy monotonicity of Proposition \ref{modifiedmono-1} the uniform $\gk$-noncollapsing estimate follows from Perelman's original argument (\cite{Perelman}, cf. \cite{GRFBook} Theorem 6.29).  Now assume we have a solution to the normalized flow on $[0,\infty)$ which satisfies a uniform curvature estimate.  By the uniform $\gk$-noncollapsing it follows that the volumes of unit balls are uniformly bounded below along the flow, and there is a uniform lower bound on the injectivity radius.  Since the volume form is bounded above pointwise by $F_t^m$ by (\ref{volume}), the total volume is bounded above by $\int_M F_t^m = (2 \pi c_1(M, J))^m$.  It follows that there is a uniform upper bound for the diameter along the flow.  With these estimates in place, it follows from Cheeger-Gromov compactness theory (cf. \cite{HamiltonCompactness} Theorem 2.3 for our simplified case) that for any sequence $\{t_j\} \to \infty$ there exists a subsequence, still denoted $\{t_j\}$, such that $(M, g_{t_j})$ converges in Cheeger-Gromov sense to a smooth limiting Riemannian manifold $(M, g_{\infty})$, where the limiting manifold is still $M$ due to the diameter upper bound.  Also, it follows from (\cite{StreetsCRF}, \cite{GRFBook} Ch. 5) that $H$ and all its covariant derivatives are uniformly bounded along the flow, and thus by choosing a further subsequence $H_{t_j}$ converges to $H_{\infty}$.  

We claim that by choosing a further subsequence the GK structure also converges, which requires establishing $C^{\infty}$ estimates for $I$ and $J$.  First, the $g$-norms of $I$ and $J$ are uniformly bounded since $g$ is compatible with both.  By a standard argument using that $\N^{B,I}$ preserves $I$, and that the coefficients of this connection are already converging by the discussion above, it follows that there are uniform $C^{\infty}$ estimates for $I$, and the same argument using $\N^{B,J}$ yields estimates for $J$.  It follows that we can choose a further subsequence such that $(M, g_{t_j}, I_{t_j}, J)$ converges in Cheeger-Gromov sense to a limiting GK manifold $(M, g_{\infty}, I_{\infty}, J_{\infty})$. By the uniform upper bound of $|b_t|^2_{g_t}$ established in Proposition~\ref{p:basic-estimate},  $|F_t|^2_{g_t}=m + |b_t|^2_{g_t}$ is uniformly bounded. As $F_t =-2g_t(I_t+ J_t)^{-1}$, it follows that $\left|\det(I_t+J_t)\right|$ is uniformly bounded from below by a positive constant, showing that $(g_{\infty}, I_{\infty}, J_{\infty})$ is symplectic type with corresponding symplectic form $F_{\infty}=-2g_{\infty}(I_{\infty}+ J_{\infty})^{-1}$.

Now note that, by rescaling, the result of Proposition \ref{modifiedmono-1} holds for the normalized flow, setting $\tau \equiv \tfrac{1}{2}$.  This also implies that the corresponding functional $\mu$ is monotone increasing.  Since the solution is nonsingular, there are uniform bounds on $\mu$ and all of its derivatives.  It thus follows that for an arbitrary sequences of times $\{t_j\} \to \infty$ one has
\begin{align*}
\lim_{j \to \infty} \brs{\frac{d}{dt} \mu(g_{t_j}, H_{t_j})} = 0.
\end{align*}
For such a sequence $\{t_j\}$, let $(M, g_{\infty}, I_{\infty}, J_{\infty})$ denote the GK limit constructed as above.  It follows from the entropy monotonicity formula that the limiting torsion $H_{\infty}$ must vanish, and furthermore $g_{\infty}$ is a shrinking soliton.  Since $H_{\infty}$ vanishes, the pairs $(g_{\infty}, I_{\infty})$ and $(g_{\infty}, J_{\infty})$ are both K\"ahler, hence the symplectic $2$-form  $F_{\infty}$ taming  $I_{\infty}$ and $J_{\infty}$  is parallel with respect to $g_{\infty}$, and hence must be a $(1, 1)$-form by the Fano condition of $(M, J_{\infty})$. It then follows that we must have $I_{\infty} = J_{\infty}$.
\end{proof}

\begin{rmk} In the toric setting, it is natural to expect that by adapting the ideas of \cite{Zhu} the convergence can be improved to show that $J_\infty = J$ and that the scalar potentials $\varphi$ converge.
\end{rmk}

\subsection{The extended Mabuchi functional}

In this section we give an extension of the Mabuchi functional to the setting of GK structures of symplectic type on a K\"ahler background.  To begin we first observe that it is possible to canonically associate a K\"ahler metric to every symplectic-type GK structure on a K\"ahler manifold.  First, using Hodge decomposition, the deRham class $\ga = [F]$ defines an element $\kal \in H^{1,1}(M, \R)$.  Using that $F$ tames $J$ and a result of Demailly-Paun~\cite{DP},  it follows that $\kal$ is a K\"ahler class of $(M, J)$. By Yau's theorem~\cite{YauCC}, we can find a unique K\"ahler metric $\kom \in \kal$ with the property
\begin{equation}\label{defining}
\kom^{[m]} = c F^{[m]}, \qquad    c =\frac{\kal^{[m]}}{[\ga]^{[m]}}.
\end{equation}

\begin{defn}\label{d:Kahler-reduction} Given $(M, J)$ K\"ahler and $\ga \in H^2(M, \mathbb R)$, let ${\bf GK}(M, J, \ga)$ be the space of symplectic-type GK structures $(g, I, J)$ such that $[F] = \alpha$.  Given $(g, I, J) \in {\bf GK}(M,J,\ga)$, the \emph{K\"ahler reduction} is the unique K\"ahler metric $\kom$ satisfying \eqref{defining}.
\end{defn}

\begin{rmk} K\"ahler reductions exist in smooth families along a solution to GKRF of symplectic type.  In general it might not be possible to derive an explicit formula for the evolution of $\kom$, but in the toric Fano case this is what is achieved in Proposition \ref{GKRF-Kahler-geometrc}.
\end{rmk}

We now recall the definitions  of the Aubin-Mabuchi functionals $\I$ and $\I^{\rho_0}$ which are defined on the space of $\kom_0$-relative K\"ahler  potentials $\varphi$ by
 \begin{equation}\label{Aubin-Mabuchi}
 \begin{split}
 (\delta_{\varphi} \I ) (\dot \varphi) &=\int_M \dot \varphi \kom_{\varphi}^{[m]},  \hspace{1.2cm} \qquad \I(0)=0, \\
  (\delta_{\varphi} \I^{\rho_0} ) (\dot \varphi)  &=\int_M \dot \varphi \rho_0 \wedge \kom_{\varphi}^{[m-1]},  \qquad \I^{\rho_0}(0)=0
  \end{split} \end{equation}
where $\rho_0$ is the Ricci form of $\kom_0$. As $\rho_0 \in 2\pi c_1(M, J)$, 
letting  $a:= 4\pi \frac{c_1(M, J) \cdot [\kal]^{[m-1]}}{[\kal]^{[m]}}$ the expression
$a \I(\varphi) - 2\I^{\rho_0}(\varphi)$  does not change if we add to $\varphi$ a real constant, i.e. it introduces a functional, denoted by $\left(a\I -2\I^{\rho_0}\right)(\kom_{\varphi})$, acting on the K\"ahler metrics $\kom_{\varphi} \in \kal$.

\begin{defn}\label{d:MabuchiGK} Fix $(M, J)$ K\"ahler and $\ga \in H^2(M, \mathbb R)$ such that $\bar{\ga} \in H^{1,1}(M, \mathbb R)$ is a K\"ahler class with representative $\kom_0 \in \bar{\ga}$.  Fix $(g, I, J) \in {\bf GK}(M,J,\ga)$ and let $\kom_{\varphi}$ denote its K\"ahler reduction.  Define the \emph{extended Mabuchi energy} by
 \begin{equation}\label{Mabuchi-CT}
 \begin{split}
 \M(F) &:= \left(\frac{\kal^{m}}{[\ga]^{m}}\right)\int_M \left[ \log\left(\frac{F^{[m]}}{\kom_0^{[m]}}\right) + \log\left(\frac{F^{[m]}}{\omega_J^{[m]}}\right)\right]F^{[m]}  + \left(a \I- 2\I^{\rho_0}\right)(\kom_{\varphi}),  \\
 a &:= 4\pi \frac{c_1(M, J) \cdot \kal^{[m-1]}}{\kal^{[m]}}.
 \end{split}
 \end{equation}
\end{defn}

\begin{rem}\label{M-inequality} In the case when $\alpha = \kal$ is a K\"ahler class, one has that $\kom_{\varphi}^{[m]}=F^{[m]} \ge \omega_J^{[m]}$.  This implies the basic inequality
\begin{equation}\label{M-tames}
\M(F) \ge \M(\kom_{\varphi}),
\end{equation}
 where $\M(\kom_{\varphi})$ is the usual Mabuchi energy of the K\"ahler metric $\kom_{\varphi} \in \alpha$, see e.g. \cite{CT}.
 \end{rem}

\subsection{The extended Mabuchi energy in the toric case}

On a compact toric K\"ahler manifold  $(M, J, \Fo, \T)$ with Delzant polytope $(\Pol, {\La})$, 
Donaldson~\cite{donaldsonJDG-02} found concise expressions for the Futaki invariant and the Mabuchi energy in terms of $(\Pol, \La)$ and the space $\cS(\Pol, \La)$. We use this point of view in order to express the extended Mabuchi energy as a functional acting on the data $(u, A, B)$ associated to $(\Pol, \La)$.  Following~\cite{donaldsonJDG-02}, we introduce a linear functional defined on the space of continuous functions on $\Pol$  by
\begin{equation}\label{Futaki}
{\bf F} (f) :=  - a\int_{\Pol} f(x)\d x  + 2 \int_{\partial \Pol} f(x) \d\sigma_{\La}, \qquad a:= 2\frac{\Vol\left({\partial \Pol}, \d\sigma_{\La}\right)}{\Vol(\Pol, \d x)},
                                              \end{equation}
where,  in the above expression,   $\d x$ is a fixed Lebesgue measure on $\tor^*$ (associated with a chosen basis of $\tor$) and $\d\sigma_{\La}$ is the measure induced by $\d x$ and the inward normal $\d L_j$ on each facet $\Pol_j \subset \partial \Pol$. If we take $f(x)= \langle \xi, x \rangle + a$ to be affine-linear, \eqref{Futaki} computes, up to a dimensional multiplicative constant,  the Futaki invariant of the holomorphic vector $JK_{\xi}$  on $(M, J,  [\Fo])$.

\begin{lemma}\label{Mabuchi-toric}  Let $(g, I, J)$ be a toric GK structure, corresponding to the data $(u, A, B)$ with respect to $\Fo$. Then 
\[\frac{1}{(2\pi)^m}\M(F)= {\bf F}(u)  - \int_{\Pol} \log \det \Big({\rm Hess}(u) + \sqrt{-1}B\big) \d x  + \int_{\Pol} \log \det \Big({\rm Hess}(u_0)\big) \d x,  \]
where $u_0$  is  the symplectic potential of a background toric K\"ahler metric. \end{lemma}
 \begin{proof} On a toric manifold $H^2(M, \C)=H^{1,1}(M, \C)$, so we are in the case $\alpha=\kal$ with $F^{[m]}=\kom_{\varphi}^{[m]}$.  By Definition \ref{d:MabuchiGK} and \eqref{potential-computation}, we compute
  \[
  \begin{split}
  \M(F) &= \M(\kom_{\varphi})  +(2\pi)^m \int_{\Pol} \log\left(\frac{\det \Hess(u)}{\det(\Hess(u) + \sqrt{-1}B)}\right) \d x.
                                             \end{split}\]
The claim follows by the above, taking in mind that $\kom_{\varphi} = \d \d^c_{J_A} \phi(\tilde y)$, where $\phi$ is the Legendre transform of $u\in \cS(\Pol, \La)$ (see \eqref{Legendre}) and $\tilde y$ are the pluriharmonic coordinates of $J_A$ (see \eqref{tilde-y-inv}), and the fact that (see \cite{donaldsonJDG-02})
 \[\M(\kom_{\varphi}) = (2\pi)^m\left({\bf F}(u) -\int_{\Pol} \log \det(\Hess(u)) \d x + \int_{\Pol} \log \det(\Hess(u_0))\d x \right) .\]
 \end{proof}
 
 We next derive fundamental variational properties for the extended Mabuchi energy.  Recall our expression for the metric
\begin{align*}
{\bf X} = \Re \big[{\rm Hess}(u) + \sqrt{-1}B\big]^{-1}.
\end{align*}
It turns out that the quantity
\begin{equation}\label{GKextremal}
 \kappa(u, A, B):= -\sum_{i.j=1}^m {\bf X}_{ij,ij} 
 \end{equation}
is identified in~\cite{boulanger,W3} with the momentum map for the action of the group of $\T$-equivariant Hamiltonian symplectomorphisms of $(M, \Fo)$ on the space of compatible toric GK structures, and thus  $\kappa$ extends the notion of scalar curvature to the GK context, via the momentum map picture of \cite{Donaldson, fujiki}.  We call $\gk$ the \emph{generalized scalar curvature}.  This point of view is extended to GK manifolds of symplectic type (without toric symmetry) in \cite{Goto-GIT}.  We thus conclude  that the critical points  of ${\bf M}$ correspond to toric GK structures for which ${\rm Gscal}(g)=\kappa(u, A, B)=a$.

\begin{prop}\label{p:Mabuchi-derivative} Given a one-parameter family of symplectic potentials $u_s = u + s \dot{u}$ and fixed matrices $A,B$, we obtain
\[ (\delta_{u} {\bf M})(\dot u) = \int_{\Pol} \dot u \left(\gk(u,A,B) - a\right) \d x.\]
In particular, the critical points of $\M$ on the space of toric generalized K\"ahler structures of symplectic type on $(M, J)$ with fixed cohomology class and holomorphic Poisson tensor have constant generalized K\"ahler scalar curvature.  Furthermore,
\[(\delta^2_{u} {\bf M})(\dot u, \dot u) = \int_{\Pol}  {\rm tr}\Big[\Big(\big({\rm Hess}(u) + \sqrt{-1}B\big)^{-1}{\rm Hess}(\dot u)\Big)^2\Big] \d x.\]
\end{prop}

\begin{proof} These follow by observing that the linearization of the entropy term is 
\[\int_{\Pol} {\rm tr} \Big[ \big({\rm Hess}(u) + \sqrt{-1}B\big)^{-1} \circ {\rm Hess}(\dot{u})\Big] \d x =  \int_{\Pol} {\rm tr}\big({\bf X} \circ {\rm Hess}(\dot{u})\big) \d x.\]
Using that ${\bf F}$ is linear, and differentiating one more time the above expression yields the formula for the Hessian of ${\bf M}$. To obtain the first formula, we need to integrate by parts the RHS of the above expression as in \cite{donaldsonJDG-02}. To this end, we observe that ${\bf X}= \big[{\rm Hess}(u) + B \big({\rm Hess}(u)\big)^{-1}B\big]^{-1}$ is a symmetric matrix on $\Pol$,  which satisfies the first order  boundary conditions of \cite[Prop. 1]{ACGT2} (see also Remark 4 in that reference).  The second variation follows easily by passing another $s$ derivative through the variation formula above and using that $\ddot{u} = 0$.
 \end{proof}

Lastly in this subsection we establish the monotonicity of the extended Mabuchi energy along NGKRF in the toric Fano setting.  We build this up through a series of lemmas.

 \begin{lemma}\label{flow-M} On a toric Fano manifold, along the flow \eqref{unnormalized} one has
  \[ \frac{d}{dt}\M(F_{t}) = \left[-\int_M |d \dot \varphi_t|^2_{\kom_{t}} F_{t}^{[m]} + \int_M \left(s_J^B(\pg_{\varphi_t})  - 2m\right)F_{t}^{[m]}\right],\]
where $s_J^B(\pg_{t})$ denotes the Bismut scalar curvature of the Hermitian structure $(\pg_{t}, J)$.
\end{lemma}
  
\begin{proof} Referring to \eqref{Mabuchi-CT} (with $\alpha=\kal$ and $a=2m$ according to our normalization),  $F_t^m =\kom_t^m$ and using \eqref{unnormalized}-\eqref{geometric-evolution} (with $c(t) \equiv 0$), we compute 
\[\begin{split}
\frac{d}{dt} \M(F_t) =&\ 2 \int_M \tr_{\pom_t}\left(\rho^B_J(\pg_t) - \kom_t\right)\kom_t^{[m]} +\int_M\big(\dot \varphi_t - 2\varphi_t + 2h_0\big) (\d\d^c\dot \varphi_t) \wedge \kom_t^{[m-1]}  \\
                                                     &\ \qquad + 2m \int_M \dot \varphi_t \kom_t^{[m]} - 2\int_M \dot\varphi_t \rho_0 \wedge \kom_t^{[m-1]}\\
                                                   =&\ \int_M\left(s_J^B(\pg_t)-  2m\right)\kom_t^{[m]} -\int_M|d\dot \varphi_t|^2_{\kom_t}\kom _t^{[m]} \\
                                                   &\ \qquad -2\int_M \dot\varphi_t(\kom_t - \kom_0)\wedge \kom_t^{[m-1]} + 2\int_M\dot \varphi_t(\rho_0-\kom_0)\wedge \kom_t^{[m-1]}  \\
                                                   &\ \qquad +2m \int_M \dot \varphi_t \kom_t^{[m]} - 2\int_M \dot\varphi_t \rho_0 \wedge \kom_t^{[m-1]}  \\
                                                 =&\ \int_M\left(s_J^B(\pg_t)-  2m\right)\kom_t^{[m]} -\int_M|d\dot \varphi_t|^2_{\kom_t}\kom_t^{[m]}.
\end{split}\]  
The claim follows by using $\kom_t^{[m]} =F_t^{[m]}$ again. \end{proof}

\begin{lemma}\label{l:key} Suppose $(M, g, I, J)$ is a compact GK manifold of symplectic type. Then
\[\int_M\left[\tr_{\omega_J}(\rho^B_J) - \tr_{F}(\rho^B_J)\right]F^{[m]} = -\frac{1}{8}\int_M\left[ \Big|\d\log\det(I+J)-\theta_I-\theta_J\Big|^2_g+|\theta_I-\theta_J|^2_g \right]F^{[m]}.\]
\end{lemma}
\begin{proof} We denote by $s_J^B$ and $s_I^B$ the Bismut-Ricci scalar curvatures of $(g, J)$ and $(g, I)$, respectively. By Proposition~\ref{p:Ricci-difference} and \cite{IP}, we know that
\[ \rho^B_J- \rho^B_I = \tfrac{1}{2}\cL_{(\theta_J^{\sharp}-\theta_{I}^{\sharp})} F, \qquad s_J^B - s_I^B = -\delta^g(\theta_J - \theta_I). \]
Using that $F=-2g(I+J)^{-1}$  and the above formula, we compute
\[
\begin{split}
&2\left[\tr_{\omega_J}(\rho^B_J) - \tr_{F}(\rho^B_J)\right] = s_J^B - \big\langle \rho^B_J, \omega_J + \omega_I\big\rangle_g = \tfrac{1}{2} s_J^B - \big\langle \rho^B_J, \omega_I\big\rangle_g \\
& \qquad = \tfrac{1}{2} (s_J^B -  s_I^B) - \tfrac{1}{2}\Big\langle \cL_{(\theta_J^{\sharp}-\theta_{I}^{\sharp})} F, \omega_I\Big\rangle_g
= -\tfrac{1}{2}\delta^g(\theta_J - \theta_I)- \tfrac{1}{2}\Big\langle \cL_{(\theta_J^{\sharp}-\theta_{I}^{\sharp})} F, \omega_I\Big\rangle_g.
\end{split} \]
We next note that, by \eqref{basic}, 
\[\cL_{(\theta_J^{\sharp}-\theta_{I}^{\sharp})} F = \d\left(F(\theta_J^{\sharp}-\theta_I^{\sharp})\right) = \d\left(I(\theta_J - \theta_I) + Ib(\theta_J^{\sharp}-\theta_I^{\sharp})\right).\]
Thus
\[
\begin{split}
2\left[\tr_{\omega_J}(\rho^B_J) - \tr_{F}(\rho^B_J)\right] &=  -\tfrac{1}{2}\delta^g(\theta_J - \theta_I)- \tfrac{1}{2}\Big\langle \d \left(I\left( (\theta_J - \theta_I) + b(\theta_J^{\sharp}-\theta_I^{\sharp})\right)\right), \omega_I\Big\rangle_g \\
&= \tfrac{1}{2}\left(\delta^g\left(b(\theta_J^{\sharp}-\theta_I^{\sharp})\right) + \big\langle \theta_J - \theta_I + b(\theta_J^{\sharp}-\theta_I^{\sharp}), \theta_I\big\rangle_g\right),
\end{split}\]
where we have used the identity  $\langle \d I \alpha, \omega_I\rangle_g= -\delta^g \alpha - \langle \alpha, \theta_I\rangle_g$.
We next use the identity  (which follows from Lemma~\ref{symplectic-type})
\[ \d \log \det (I+J) = \theta_I + \theta_J + b(\theta_J^{\sharp}-\theta_I^{\sharp})\]
 in order to express the  terms containing $b(\theta_J^{\sharp}-\theta_I^{\sharp})$. We thus get
 \[
\begin{split}
2\left[\tr_{\omega_J}(\rho^B_J) - \tr_{F}(\rho^B_J)\right] &=   \tfrac{1}{2}\left(-\Delta^g \log \det (I+J) - \delta^g(\theta_I +\theta_J)  + \big\langle - 2\theta_I + \d \log \det (I+J), \theta_I\big\rangle_g\right) \\
&= \tfrac{1}{2}\left(-\Delta^g \log \det (I+J) - \delta^g(\theta_I +\theta_J)  - 2|\theta_I|^2_g + \big\langle \d \log \det (I+J), \theta_I\big\rangle_g\right).
\end{split}\]
We finally use that (as $F= -2g(I+J)^{-1}$)
\begin{equation*}
 F^{[m]} = 2^m \det(I+J)^{-\frac{1}{2}}\,  \omega_J^{[m]} \end{equation*}
and integration by parts to get
\[
\begin{split}
&\int_M\left[\tr_{\omega_J}(\rho^B_J) - \tr_{F}(\rho^B_J)\right] F^{[m]} \\
&=   2^{m-2}\int_M\left(-\Delta^g \log \det (I+J) + \big\langle \d \log \det (I+J), \theta_I\big\rangle_g - \delta^g(\theta_I +\theta_J)  - 2|\theta_I|^2_g \right)\det(I+J)^{-\frac{1}{2}}\,  \omega_J^{[m]}  \\
&=   2^{m-2}\int_M\left(-\Delta^g \log \det (I+J) + \big\langle \d \log \det (I+J), \theta_I\big\rangle_g  - |\theta_I|^2_g-|\theta_J|^2_g-2\delta^g\theta_J \right)\det(I+J)^{-\frac{1}{2}}\,  \omega_J^{[m]}  \\
&= 2^{m-2}\int_M\left[-\frac{1}{2}\Big|\d\log \det(I+J)\Big|^2_g + \big\langle \d \log \det (I+J), \theta_I+\theta_J\big\rangle_g-|\theta_I|^2_g-|\theta_J|^2_g \right]\det(I+J)^{-\frac{1}{2}}\,  \omega_J^{[m]}   \\
&= -\frac{1}{8}\int_M\left[ \Big|\d\log\det(I+J)-\theta_I-\theta_J\Big|^2_g+|\theta_I-\theta_J|^2_g \right]F^{[m]}
\end{split}
\]
where we have used the relation $\delta^g \theta_I + |\theta_I|^2_g =\frac{1}{6}|H|_g^2=\delta^g\theta_J+|\theta_J|^2_g$ from \cite{IP}.
\end{proof}

 \begin{cor}\label{Mabuchi-monotone} On a toric Fano manifold, the Mabuchi functional $\M(F_{\varphi_t})$ is monotone decreasing along NGKRF.
 \end{cor}
\begin{proof} This follows  directly  from Lemma~\ref{flow-M}  and Lemma~\ref{l:key}, noting that $\tr_{\pom}(\rho^B_J(g))= \frac{s_J^B}{2}$ and 
\[\int_M \tr_{F}(\rho_J^B(g)) F^{[m]} =\int_M \rho^B_J(g)\wedge F^{[m-1]} = m\]
by the Fano condition and the normalization $[F]\in 2\pi c_1(M, J)$. \end{proof}

\subsection{Weak convergence}\label{s:weakconvergence}

We suppose through this section that $(M, J, \T)$ is a smooth toric Fano manifold,   $\kom_0 \in 2\pi c_1(M, J)$ is an initial K\"ahler form with corresponding  barycentered Fano polytope $(\Pol, {\La})$, and we further assume that  the Futaki invariant of $(M, J, [\kom_0])$ vanishes. In this case, by~\cite{WZ},  $(M, J)$ admits a K\"ahler-Einstein metric and,  by \cite{DR,Tian}, the Mabuchi energy is coercive relative to $\Aut_{0}(M, J)$.  We first recall what this precisely means.  Denote by 
\[ \cH(M, \kom_0)=\left\{ \varphi \in C^{\infty}(M)\, | \, \kom_{\varphi}:= \kom_0 + \d\d^c \varphi >0 \right\}\] the space of  $\kom_0$-relative K\"ahler potentials  on $M$, and  consider the Aubin-Mabuchi functional $\I$ on $\cH(M, \kom_0)$, defined in \eqref{Aubin-Mabuchi}. Notice that $\I(\varphi + c) = \I(\varphi) + c\Vol(M, \kom_0)$, so we can introduce the slice
\[\mathring{\cH}(M, \kom_0) := \cH(M, \kom_0) \cap \I^{-1}(0),\]
and consider the induced action of $\G:=\Aut_0(M, J)$ on $\mathring{\cH}(M, \kom_0)$ via the natural action of $\G$ on the space of $\T$-invariant K\"ahler metrics in $[\kom_0]$ by pullbacks, i.e. for any $\varphi \in \mathring{\cH}(M, \kom_0)$ and any $\tau \in \G$, we let $\tau[\varphi]$ denote the uniquely determined $\kom_0$-relative K\"ahler potential in $\mathring{\cH}(M, \kom_0)$ of the K\"ahler metric $\tau^*(\omega_{\varphi})$. We thus have
\[\tau[\varphi] = \tau[0] + \varphi \circ \tau. \]

We next consider the $d_1$-distance on $\cH(M, \kom_0)$,  introduced via the $L^{1}$-length of a smooth curve $\varphi(t )\in \cH(M, \kom_0)$:
\[l_1(\varphi(t)) = \int_{0}^1 \int_M |{\dot \varphi}(t)| \, (\kom_{\varphi(t)})^{[m]}.\]
By the results in \cite{Da}, 
\[d_1(\varphi_0, \varphi_1) := \inf \Big\{l_1(\varphi(t))\,  \Big|  \, \varphi(t, x) \in C^{\infty}([0, 1] \times M),  \, \varphi(t, \cdot ) \in  \cH(M, \kom_0), \, \varphi(0)=\varphi_0, \, \varphi(1)=\varphi_1 \Big\}\]
defines a distance. 
We then have
\begin{thm}\label{DR}\cite{DR} If $(M, J, [\kom_0])$ admits a K\"ahler-Einstein metric, then the Mabuchi energy $\M$  is $\G$-invariant and $\G$-coercive  on $\cH(M, \kom_0)$,   i.e.  there are uniform constants $\lambda>0, \delta$ such that 
\[ \M (\varphi) \ge  \lambda \inf_{\tau \in \G} d_1(0, \tau[\varphi]) - \delta, \qquad \forall \varphi \in \mathring{\cH}(M, \kom_0). \]
\end{thm}
An important ingredient in proving the above result is the weak compactness, established in \cite{BBEGZ,BDL}. To state it, we denote by 
$\PSH(M, \kom_0)$ the space of $\kom_0$-relative plurisubharmonic functions on $M$, i.e. the usc functions $\varphi \in L^1(M, \kom_0)$ such that $\kom_{\varphi}=\kom_0 + \d\d^c \varphi \ge 0$ in the sense of currents. Furthermore, we denote by $\E(M, \kom_0) \subset \PSH(M, \kom_0)$ the subspace of elements of \emph{full Monge-Amp\`ere mass}, i.e.
\[ \E(M, \kom_0) = \Big\{ \varphi \in \PSH(M, \kom_0) \, \Big| \, \int_M \kom_{\varphi}^{[m]} = \int_M \kom_0^{[m]}=\Vol(M, \kom_0)\Big\}.\]
By the results in \cite{GZ,BK},  each $\varphi \in \E(M, \kom_0)$ can be weakly approximated (as a current) by a sequence of smooth functions $(\varphi_j)_j \in \cH(M, \kom_0)$ and for any such sequence $\lim_{j\to \infty} \kom_{\varphi_j}^{[m]} = \kom_{\varphi}^{[m]}$ in the sense of measures. We can take the latter as a definition of $\kom_{\varphi}^{[m]}$ when $\varphi \in \E(M, \kom_0)$. Following \cite{GZ}, one further introduces the sub-space of elements with \emph{full Monge-Amp\`ere mass and finite energy}
\[\E^1(M, \kom_0) = \Big\{ \varphi \in \E(M, \kom_0) \, \Big| \, \int_M |\varphi| \, \kom_{\varphi}^{[m]} < \infty \Big\}.\]
The central fact~\cite{Da} in this theory is that $d_1$ extends to define a \emph{complete} distance on $\E^1(M, \kom_0)$, such that $(\E^1(M, \kom_0), d_1)$ is a complete geodesic space in which $\cH(M, \kom_0)$ is densely embedded,  and the $d_1$-convergence on $\E^1(M, \kom_0), d_1)$ is stronger than both the $L^1(M, \kom_0)$ and the weak convergence of $(1,1)$-currents.  Furthermore, $\M$ naturally extends to a continuous and  lsc  functional   on $\E^1(M, \kom_0)$.  There is a key compactness principle:
\begin{thm}\label{compactness}\cite{BBEGZ, BDL} If $(\varphi_j)_j \in \E^1(M, \kom_0)$ is a sequence with
\[ d_1(0, \varphi_j) <C, \qquad \M(\varphi_j) < C,\]
then there exists a subsequence $(\varphi_{j_k})_k$ which converges  with respect to $d_1$ to an element $\varphi_0 \in \E^1(M, \kom_0)$.
\end{thm}

In the toric case considered in this paper,  we  shall rather work with the subspace
\[\cH_{\T}(M, \kom_0) := \cH(M, \kom_0) \cap C^{\infty}_{\T}(M) \]
of $\T$-invariant relative K\"ahler potentials,  and consider instead of $\G = \Aut_0(M, J)$  the complex torus $\T_{\C}$. The coercivity principle of \cite[Theorem~3.4]{DR} still applies in the $\T$-relative setting (see \cite{hisamoto,ZZ}):
\begin{thm}\label{DT-toric} If $(M, J, [\kom_0], \T)$ admits a K\"ahler-Einstein metric, then the Mabuchi   energy $\M : \cH_{\T}(M, \kom_0) \to \R$ is $\T_{\C}$-invariant and $\T_{\C}$-coercive,  i.e. there are uniform constants $\lambda>0, \delta$ such that 
\[ \M (\varphi) \ge  \lambda \inf_{\tau \in \T_{\C}} d_1(0, \tau[\varphi]) - \delta, \qquad \forall \varphi \in \mathring{\cH}_{\T}(M, \kom_0). \]
\end{thm}

We next use Theorem~\ref{DT-toric} to obtain weak convergence of the global solution of NGKRF in the toric K\"ahler-Einstein case.

\begin{proof}[Proof of Theorem~\ref{t:convthm}] 
Let $\varphi_t := \phi_t - \phi_0 \in \cH_{\T}(M, \omega_0)$ be the relative K\"ahler potentials along the K\"ahler reduction \eqref{unnormalized} of the NGKRF, and $\mathring{\varphi}_t \in \mathring{\cH}_{\T}(M, \kom_0) ={\bf I}^{-1}(0)$ the corresponding normalized $\kom_0$-relative K\"ahler potentials.  Using Theorem~\ref{DT-toric}, Corollary~\ref{Mabuchi-monotone} and \eqref{M-tames}, we have along the K\"ahler reduction \eqref{unnormalized} of the normalized generalized K\"ahler Ricci flow,
\[ \lambda \inf_{\tau \in \T_{\C}} d_1(0, \tau[\mathring{\varphi}_t]) \leq  \M(\kom_{\varphi_t}) + \delta  \leq \M(F_{\varphi_t}) + \delta \leq  {\M}(F_0)+ \delta.\]
Furthermore, by  \cite[Prop.~6.8]{DR} (using that  $\T_{\C}$ is reductive), for each $t\in [0, \infty)$ there exists  $\tau_t \in \T_{\C}$ such that 
\[\inf_{\tau \in \T_{\C}} d_1(0, \tau[\mathring{\varphi}_t])= d_1(0, \tau_t[\mathring{\varphi}_t]). \]
We thus get uniform estimates
\[ d_1(0, \tau_t[\mathring{\varphi}_t])  \leq C, \qquad \M(\tau_t[\mathring{\varphi}_t])= \M(\varphi_t) < C \]
which, by Theorem~\ref{compactness} and the fact that $\E^1_{\T}(M, \kom_0)$ is $d_1$-closed as a subset of $\E^1(M, \kom_0)$~\cite{DR}, proves the theorem. 
\end{proof}

\end{document}